\documentclass[a4paper,12pt]{amsart}
\usepackage[T1]{fontenc}
\usepackage[utf8]{inputenc}
\usepackage[margin = 2.5cm]{geometry}
\usepackage{amsmath}
\usepackage{amssymb}
\usepackage{amsthm}
\usepackage{fancyhdr}
\usepackage{geometry}
\usepackage{lmodern}
\usepackage{enumerate}
\usepackage{yfonts}
\usepackage{latexsym}
\usepackage{mathtools}
\usepackage{yhmath}
\usepackage{tikz}
\usepackage{array}
\usepackage{color}
\usepackage{todonotes}
\usepackage{mathrsfs}
\usepackage[all]{xy}
\usepackage{ytableau}
\usepackage{multicol}	
\usepackage{caption}  
\usepackage[sc]{mathpazo}
\usepackage[normalem]{ulem}
\newcolumntype{L}[1]{>{\raggedright\let\newline\\\arraybackslash\hspace{0pt}}m{#1}}
\newcolumntype{C}[1]{>{\centering\let\newline\\\arraybackslash\hspace{0pt}}m{#1}}
\newcolumntype{R}[1]{>{\raggedleft\let\newline\\\arraybackslash\hspace{0pt}}m{#1}}
\usepackage[pdfencoding=unicode,
            psdextra,
            colorlinks=true,
            linkcolor=blue,
            anchorcolor=blue,
            citecolor=blue,
            filecolor=blue,
            urlcolor=blue]{hyperref}

\theoremstyle{definition}
\newtheorem{defi}{Definition}[section]
\newtheorem{ex}[defi]{Example}
\newtheorem{rmk}[defi]{Remark}

\theoremstyle{plain}
\newtheorem{thm}[defi]{Theorem}
\newtheorem{lem}[defi]{Lemma}
\newtheorem{cor}[defi]{Corollary}
\newtheorem{prop}[defi]{Proposition}

\newtheorem*{prop*}{Proposition}

\DeclareMathOperator{\soc}{soc}

\DeclareMathOperator{\GL}{GL}

\DeclareMathOperator{\e}{e}

\newcommand{\N}{\mathbb{N}}
\newcommand{\Q}{\mathbb{Q}}

\newcommand{\Z}{\mathbb{Z}}

\newcommand{\m}{\mathfrak{m}}
\newcommand{\ml}{\mathfrak{l}}

\renewcommand{\b}{\mathrm{b}}

\newcommand{\lr}[1]{\left(#1 \right)}

\usepackage{caption}
\usepackage[
backend=bibtex,
style=alphabetic,
doi=false,isbn=false,url=false,eprint=false
]{biblatex}

\title[Mixed RSK for non-Archimedean $\GL_n$]{A Mixed Robinson--Schensted--Knuth Construction in the Representation Theory of $\GL_n$ over Local non-Archimedean Fields}
         \date{\today}
	\author{Ilse Fischer \and Markus Reibnegger}
	\thanks{The authors acknowledge support from the Austrian Science Fund (FWF) grant 10.55776/F1002.}

\begin{document}
\begin{abstract}
		We present a simple combinatorial method for computing the socle of representations of $\GL_n(F)$, where $F$ is a local non-Archimedean field, parabolically induced from two ladder representations satisfying a permissibility condition. We adapt recent approaches to such problems using the classical Robinson--Schensted--Knuth correspondence (RSK) and introduce a new variant of RSK that we call mixed-RSK. Some combinatorial properties of this map are investigated and then used to resolve a conjecture of Erez Lapid. 
\end{abstract}
	
	\maketitle
	
	\tableofcontents
	\section{Introduction}
	
	\subsection{Motivation}
	
	The representation theory of classical groups over non-Archimedean local fields has played a central part in many developments of the previous decades in number theory. Of particular importance is the representation theory of the groups $\GL_n(F)$ of invertible $n\times n$-matrices over a local non-Archimedean field $F$. By work of Bernstein and Zelevinsky \cite{BZ76,BZ77,Zelevinsky80}, the (complex and smooth) irreducible representations of $\GL_n(F)$ can be parametrized by so-called \textit{multisegments}, which are essentially combinatorial objects. For many problems concerning representations of $\GL_n(F)$ (precisely speaking: those concerning only representations with cuspidal support in a single cuspidal line) it is in fact enough to consider multisegments as purely combinatorial objects: a multisegment $\m$ can then be seen as a finite multiset of integer intervals $[a,b]=\{x\in\Z\colon a\leq x\leq b\}$. For each such multisegment there is a unique irreducible representation, denoted $Z(\m)$, of a general linear group and every irreducible representation of $\GL_n(F)$ is of this form. Having such simple combinatorial objects to classify irreducible representations, it is natural to desire (relatively) simple combinatorial algorithms for computations involving such representations, see \cite{KL12,LM14,LM16} for some developments on this. In analogy to the classical Littlewood--Richardson rule for the symmetric group $S_n$, a natural and interesting, but difficult, problem is giving a combinatorial description for all irreducible subquotients of $Z(\m)\times Z(\mathfrak{n})$ where $\m,\mathfrak{n}$ are multisegments and $\times$ denotes normalised parabolic induction. A complete understanding of this (or even a necessary and sufficient condition for $Z(\m)\times Z(\mathfrak{n})$ being irreducible) has remained out of reach, see \cite{LM16,AL25,LM25} and the discussion therein for results in this direction. Nevertheless, for certain classes of multisegments, so-called \textit{ladder} multisegments, there has been considerably greater progress in giving combinatorial methods to answer such decomposition questions.
	
	Recently, Gurevich and Lapid \cite{GL21} have used a version of the classical Robinson--Schensted--Knuth (RSK) correspondence to give a different parametrisation of the irreducible representations by what they call \textit{RSK-standard} modules. They construct these modules by using the RSK correspondence to assign to any multisegment $\m$ a sequence $(\ml_1,\dots,\ml_k)$ of ladder multisegments. More precisely, for every multisegment $\m$, $Z(\m)$ is the unique irreducible subrepresentation of $Z(\ml_k)\times\dots\times Z(\ml_1)$, where $(\ml_1,\dots,\ml_k)$ is the sequence of ladders associated to $\m$ by RSK (a full proof of this, using the representation theory of quiver Hecke algebras, can be found in \cite{Gurevich2023b}). The problem of determining all irreducible subquotients of such an RSK-standard module $Z(\ml_k)\times\dots\times Z(\ml_1)$ was left open. To resolve it, one first needs a complete combinatorial understanding of the decomposition of $Z(\ml)\times Z(\ml^\prime)$ where $\ml,\ml^\prime$ are ladder multisegments. In this paper, we present a partial result in this direction by giving a simple combinatorial description of the socle of such a representation by adapting the RSK map of Gurevich and Lapid. This extends their description in \cite{GL21} in the case of a pair of ladders and also extends a similar, albeit differently formulated, description in \cite[Theorem 6.11]{Gurevich2021}. Our result still requires the ladders $\ml,\ml^\prime$ to satisfy a permissibility condition (see Definition \ref{def:perm}). Finally, we use our construction to resolve a conjecture of Erez Lapid, communicated to us in a talk, which states that for any multisegment $\m_\sigma$ naturally corresponding to a $321$-avoiding permutation $\sigma\in S_n$ there are precisely $n+1$ regular pairs $(\ml,\ml^\prime)$ of ladders such that $\soc(Z(\ml)\times Z(\ml^\prime))=Z(\m_\sigma)$, see Theorem \ref{thm:n+1} for a precise statement.
	
	\subsection{Mixed Robinson--Schensted--Knuth Correspondence}
	
	We describe the construction we found for $\soc(Z(\ml)\times Z(\ml^\prime))$ in broad terms, the precise description of this can be found in Section \ref{sec:mat}. The results of \cite{GL21} compute this socle using RSK, but this only works if the pair of ladders $(\ml,\ml^\prime)$ is \textit{dominant} (see Definition \ref{def:dominant}). To compute it also in the case of non-dominant pairs $(\ml,\ml^\prime)$ we decompose the pair by splitting-off a (maximal in a suitable sense) dominant part of it to get two pairs, one of which is dominant and the other enjoying a property dual to dominance. We reformulate the description of the inverse of the RSK map in \cite{GL21} to a matching rule between the segments of $\ml$ and $\ml^\prime$ for the dominant part and use a suitable dual version of it in the remaining part to get a set of matched segments $(\Delta,\Delta^\prime)$ with $\Delta$ in $\ml$ and $\Delta^\prime$ in $\ml^\prime$. Then the multisegment underlying $\soc(Z(\ml)\times Z(\ml^\prime))$ is obtained by replacing all such matched pairs of segments by their \textit{union-intersection} $\Delta\cup\Delta^\prime+\Delta\cap\Delta^\prime$ (see Theorem \ref{thm:soc}). We denote by $\Psi(\ml,\ml^\prime)$ the resulting multisegment. Union-intersection is a fundamental tool for decomposition problems in the representation theory of $\GL_n(F)$, see \cite[Theorem 7.1]{Zelevinsky80} for a classical and very useful example of this. We view it as a nice feature of our description that it computes $\soc(Z(\ml)\times Z(\ml^\prime))$ by isolating the pairs of segments which have to be replaced by their union-intersection since this is the simplest combinatorial description that can be hoped for. This description also works for non-dominant pairs $(\ml,\ml^\prime)$, which was needed to resolve the $(n+1)$-to-$1$ conjecture of Erez Lapid (here Theorem \ref{thm:n+1}). Our description still needs a permissibility condition on the pairs $(\ml,\ml^\prime)$ similar to that of \cite{GL21}, see Definition \ref{def:perm} for details.
	
	We briefly describe this $(n+1)$-to-$1$ result and the combinatorial construction we use to prove it. To a pair $(I,J)$ of subsets of $[n]=\{1,\dots,n\}$ of equal size $k$, we associate a ladder multisegment by setting
	\[ \ml_{I,J}=\sum_{i=1}^k[a_i,b_i+n-1], \]
	where the $a_i$ resp. $b_i$ are the elements of $I$ resp. $J$ in decreasing order. Similarly, to any permutation $\sigma\in S_n$ we associate a multisegment by setting
	\[ \m_\sigma=\sum_{i=1}^n[i,\sigma(i)+n-1]. \]
	By \cite[Corollary 4.13]{Gurevich2020}, for any pair $(I,J)$ of sets as above, there is a unique $321$-avoiding permutation $\sigma\in S_n$ such that
	\[ \soc(Z(\ml_{I,J})\times Z(\ml_{[n]\setminus I,[n]\setminus J}))=Z(\m_\sigma). \]
	This defines a map $\psi\colon (I,J)\mapsto \sigma$ from pairs of subsets of $[n]$ of equal size to $321$-avoiding permutations. Note that the number of pairs $(I,J)$ of subsets of $[n]$ of equal size is $\binom{2n}{n}$ and the number of $321$-avoiding permutations is the $n$-th Catalan number $C_n=\frac{1}{n+1}\binom{2n}{n}$ as was first proved by Knuth \cite{Knuth69}. The conjecture of Erez Lapid was that $\psi$ is $(n+1)$-to-$1$. By our results described above, this map can be computed using the map $\Psi$ that computes such socles. To show that $\psi$ is $(n+1)$-to-$1$, we explicitly describe a combinatorial method to compute $\psi^{-1}(\sigma)$ for any $321$-avoiding $\sigma$.
	
	Concretely, we construct pairs $(I,J)\in \psi^{-1}(\sigma)$ by applying a \textit{mixed-RSK} map to $\sigma$. This is a construction which may be of independent combinatorial interest and that we sketch here. Thinking of permutations $\sigma$ as a collection of points $(i,\sigma(i))$ in the plane, Viennot's geometric construction of the classical Robinson--Schensted correspondence associates a pair of standard Young tableaux of equal shape to $\sigma$ by connecting the points of $\sigma$ by so-called shadow lines, see \cite{Fulton1996} and also Figure \ref{fig:rsknw} for an illustration. Furthermore, it is a classical result that under the RSK correspondence, $321$-avoiding permutations correspond precisely to pairs of standard Young tableaux with at most $2$ rows. Such pairs can in turn clearly be identified with certain pairs $(I,J)$ of subsets of $[n]$ of equal size by specifying one of the two rows. Note that only $C_n$ of the $(n+1)C_n$ pairs of subsets $(I,J)$ arise in this way.
	
	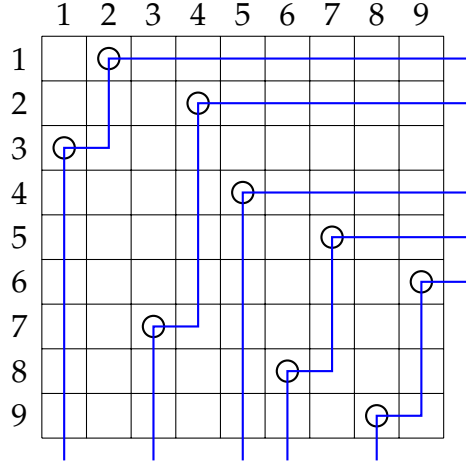
\begin{figure}[ht]
			\centering
			\begin{tikzpicture}[x=1.4em,y=1.4em, thick,color = black]
			\draw[step=1,black,thin] (0,0) grid (9,9);
			\draw [black] (1.5,8.5) circle (4pt);
			\draw [black] (3.5,7.5) circle (4pt);
    			\draw [black] (4.5,5.5) circle (4pt);
    			\draw [black] (6.5,4.5) circle (4pt);
    			\draw [black] (8.5,3.5) circle (4pt);
    			\draw [black] (0.5,6.5) circle (4pt);
    			\draw [black] (2.5,2.5) circle (4pt);
    			\draw [black] (5.5,1.5) circle (4pt);
    			\draw [black] (7.5,0.5) circle (4pt);
    			\draw [blue] (0.5,-0.5) -- (0.5,6.5) -- (1.5,6.5) -- (1.5,8.5) -- (9.5,8.5);
    			\draw [blue] (2.5,-0.5) -- (2.5,2.5) -- (3.5,2.5) -- (3.5,7.5) -- (9.5,7.5);
    			\draw [blue] (4.5,-0.5) -- (4.5,5.5) -- (9.5,5.5);
    			\draw [blue] (5.5,-0.5) -- (5.5,1.5) -- (6.5,1.5) -- (6.5,4.5) -- (9.5,4.5);
    			\draw [blue] (7.5,-0.5) -- (7.5,0.5) -- (8.5,0.5) -- (8.5,3.5) -- (9.5,3.5);
    			\node at (0.5,9.5) {$1$};
    			\node at (1.5,9.5) {$2$};
    			\node at (2.5,9.5) {$3$};
    			\node at (3.5,9.5) {$4$};
    			\node at (4.5,9.5) {$5$};
    			\node at (5.5,9.5) {$6$};
    			\node at (6.5,9.5) {$7$};
    			\node at (7.5,9.5) {$8$};
    			\node at (8.5,9.5) {$9$};
    			\node at (-0.5,8.5) {$1$};
    			\node at (-0.5,7.5) {$2$};
    			\node at (-0.5,6.5) {$3$};
    			\node at (-0.5,5.5) {$4$};
    			\node at (-0.5,4.5) {$5$};
    			\node at (-0.5,3.5) {$6$};
    			\node at (-0.5,2.5) {$7$};
    			\node at (-0.5,1.5) {$8$};
    			\node at (-0.5,0.5) {$9$};
		\end{tikzpicture}
		\caption{Shadow lines for $\sigma=241579368$ with a north-west orientation.}
		\label{fig:rsknw}
	\end{figure}  
	
	This construction naturally leaves a choice of orientation for the shadow lines. Indeed the version of RSK used in \cite{GL21} is based on an orientation where the 'light' creating the shadow lines is not in the north-west corner as in Figure \ref{fig:rsknw} but rather in the south-east corner as in Figure \ref{fig:rskse}.
	
	\begin{figure}[ht]
			\centering
			\begin{tikzpicture}[x=1.4em,y=1.4em, thick,color = black]
			\draw[step=1,black,thin] (0,0) grid (9,9);
			\draw [black] (1.5,8.5) circle (4pt);
			\draw [black] (3.5,7.5) circle (4pt);
    			\draw [black] (4.5,5.5) circle (4pt);
    			\draw [black] (6.5,4.5) circle (4pt);
    			\draw [black] (8.5,3.5) circle (4pt);
    			\draw [black] (0.5,6.5) circle (4pt);
    			\draw [black] (2.5,2.5) circle (4pt);
    			\draw [black] (5.5,1.5) circle (4pt);
    			\draw [black] (7.5,0.5) circle (4pt);
    			\draw [blue] (-0.5,0.5) -- (7.5,0.5) -- (7.5,3.5) --(8.5,3.5) -- (8.5,9.5);
    			\draw [blue] (-0.5,1.5) -- (5.5,1.5) -- (5.5,4.5) -- (6.5,4.5) -- (6.5,9.5);
    			\draw [blue] (-0.5,2.5) -- (2.5,2.5) -- (2.5,5.5) -- (4.5,5.5) -- (4.5,9.5);
    			\draw [blue] (-0.5,6.5) -- (0.5,6.5) -- (0.5,7.5) -- (3.5,7.5) -- (3.5,9.5);
    			\draw [blue] (-0.5,8.5) -- (1.5,8.5) -- (1.5,9.5);
    			\node at (0.5,-.5) {$1$};
    			\node at (1.5,-.5) {$2$};
    			\node at (2.5,-.5) {$3$};
    			\node at (3.5,-.5) {$4$};
    			\node at (4.5,-.5) {$5$};
    			\node at (5.5,-.5) {$6$};
    			\node at (6.5,-.5) {$7$};
    			\node at (7.5,-.5) {$8$};
    			\node at (8.5,-.5) {$9$};
    			\node at (9.5,8.5) {$1$};
    			\node at (9.5,7.5) {$2$};
    			\node at (9.5,6.5) {$3$};
    			\node at (9.5,5.5) {$4$};
    			\node at (9.5,4.5) {$5$};
    			\node at (9.5,3.5) {$6$};
    			\node at (9.5,2.5) {$7$};
    			\node at (9.5,1.5) {$8$};
    			\node at (9.5,0.5) {$9$};
		\end{tikzpicture}
		\caption{Shadow lines for $\sigma=241579368$ with a south-east orientation.}
		\label{fig:rskse}
	\end{figure}
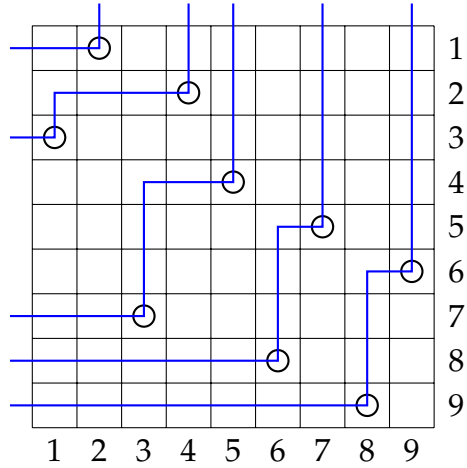 
		
	The results of \cite{GL21} and a simple duality argument show that both of these versions of the Robinson--Schensted correspondence can be used to obtain elements of $\psi^{-1}(\sigma)$, by identifying the resulting $2$-line tableaux either with their first rows in the case of Figure \ref{fig:rsknw} or with their second rows in the case of Figure \ref{fig:rskse} above. We prove in Section \ref{sec:appl} that all preimages of $\sigma$ under the map $\psi$ can be found by \textit{mixing} these two orientations. For a point $\Gamma$ in the plane, we define the mixed-RSK map $\mathcal{R}_\Gamma$ by applying the shadow construction of RSK as in Figure \ref{fig:rsknw} to all points of a $321$-avoiding permutation $\sigma$ that are south-east of $\Gamma$ and by applying the other version of RSK to all remaining points. We illustrate this in Figure \ref{fig:mixrskex}, which shows that $\mathcal{R}_{(4,5)}(241579368)=(\{1,2,4,5,6\},\{1,3,5,6,8\})$.
	
	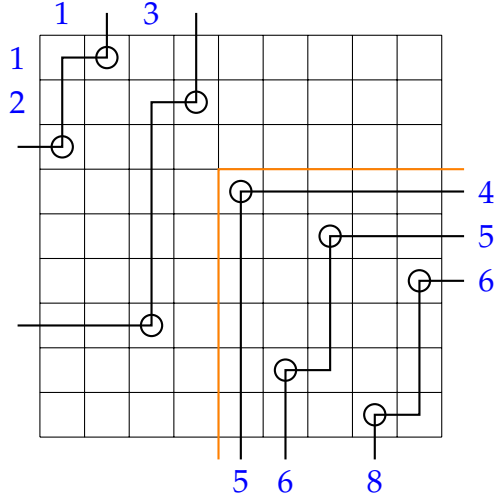
\begin{figure}[h!]
			\centering
			\begin{tikzpicture}[x=1.4em,y=1.4em, thick,color = black]
			\draw[step=1,black,thin] (0,0) grid (9,9);
			\draw [black] (1.5,8.5) circle (4pt);
			\draw [black] (3.5,7.5) circle (4pt);
    			\draw [black] (4.5,5.5) circle (4pt);
    			\draw [black] (6.5,4.5) circle (4pt);
    			\draw [black] (8.5,3.5) circle (4pt);
    			\draw [black] (0.5,6.5) circle (4pt);
    			\draw [black] (2.5,2.5) circle (4pt);
    			\draw [black] (5.5,1.5) circle (4pt);
    			\draw [black] (7.5,0.5) circle (4pt);
    			\draw [orange] (4,-0.5) -- (4,6) -- (9.5,6);
    			\draw [black] (-0.5,2.5) -- (2.5,2.5) -- (2.5,7.5) -- (3.5,7.5) -- (3.5,9.5);
    			\draw [black] (-0.5,6.5) -- (0.5,6.5) -- (0.5,8.5) -- (1.5,8.5) -- (1.5,9.5);
    			\draw [black] (4.5,-0.5) -- (4.5,5.5) -- (9.5,5.5);
    			\draw [black] (5.5,-0.5) -- (5.5,1.5) -- (6.5,1.5) -- (6.5,4.5) -- (9.5,4.5);
    			\draw [black] (7.5,-0.5) -- (7.5,0.5) -- (8.5,0.5) -- (8.5,3.5) -- (9.5,3.5);
    			\node [blue] at (-0.5,8.5) {$1$};
    			\node [blue] at (0.5,9.5) {$1$};
    			\node [blue] at (-0.5,7.5) {$2$};
    			\node [blue] at (2.5,9.5) {$3$};
    			
    			\node [blue] at (4.5,-1) {$5$};
    			\node [blue] at (10,5.5) {$4$};
    			\node [blue] at (5.5,-1) {$6$};
    			\node [blue] at (10,4.5) {$5$};
    			\node [blue] at (7.5,-1) {$8$};
    			\node [blue] at (10,3.5) {$6$};
		\end{tikzpicture}
		\caption{Illustrating the computation of $\mathcal{R}_{(4,5)}(241579368)$.}
		\label{fig:mixrskex}
	\end{figure}
	
	We prove in Section \ref{sec:appl} that these mixed RSK maps are sections of the map $\Psi$ computing socles and that for every $321$-avoiding $\sigma$ there are precisely $n+1$ different pairs $\mathcal{R}_\Gamma(\sigma)$ as $\Gamma$ ranges over all elements of $\Z^2$.
	
	So far we have not investigated these ideas further, but the combinatorics of the mixed-RSK construction might be interesting in its own right. A natural question is whether such a construction can be extended in a useful way to all permutations without the pattern avoidance condition.
	
	\subsection{Outline of the Paper}
	In Section \ref{sec:prelim} we introduce the necessary notation and some results we use from the representation theory of $\GL_n(F)$ and further establish some results on pairs of ladder multisegments. We define the map $\Psi$ in Section \ref{sec:mat} and prove in Section \ref{sec:mainthm} that, for pairs of ladders $(\ml,\ml^\prime)$, it computes $\soc(Z(\ml)\times Z(\ml^\prime))$. In Section \ref{sec:appl} we connect this to RSK combinatorics, introduce the mixed-RSK maps and prove that the map $\psi$ mentioned above is $(n+1)$-to-$1$.
	
	\subsection{Acknowledgements}
		We want to thank Erez Lapid for telling us about his conjecture (Theorem \ref{thm:n+1}) and for interesting discussions as well as reading a draft of this article. We also thank Alberto Mínguez for his feedback on earlier drafts of this article.
		
	\section{Preliminaries}\label{sec:prelim}
	
		\subsection{Multisegments}
		
		We begin by introducing the combinatorial objects we will be concerned with and explain their basic properties.
		
		\begin{defi}
			A \textit{segment} is an integer interval $\Delta=[c,d]\subseteq\Z$ with $c\leq d$. For such a segment we write $\b(\Delta)=c$, $\e(\Delta)=d$. We call $\Delta$ a \textit{singleton} if $\b(\Delta)=\e(\Delta)$.
			
			A \textit{multisegment} $\m$ is a finite multiset of segments. We always write such multisets additively, so that a multisegment $\m$ can be written as a formal sum
			\begin{align*}
				\m=\sum_{i\in I}\Delta_i,
			\end{align*}
			where $I$ is a finite set and $\Delta_i$ are segments. We denote the empty multisegment, where $I=\emptyset$, by $0$.
		\end{defi}
		
		We often identify segments $\Delta$ with their pair of endpoints $(\b(\Delta),\e(\Delta)) \in \Z^2$ and hence view multisegments as finite multisets of pairs of integers. Throughout the article, we use matrix coordinates (row, column) on $\Z^2$, with coordinates increasing top-to-bottom and left-to-right. Since all segments satisfy $\b(\Delta)\leq\e(\Delta)$, the identification only takes place in the region $T\coloneqq\{(x,y)\in\Z^2\colon x\leq y\}$. Furthermore, we will frequently use cardinal directions (north, east, south, west) and their combinations to indicate the relative position of points, with capitalisation indicating a strict relation. As an example, for points $(a,b),(c,d)\in \Z^2$ saying that $(a,b)$ is northWest (abbreviated as nW) of $(c,d)$ is equivalent to saying $a\leq c$ and $b<d$. See also Example \ref{ex:mw}.
		
		We order segments $\Delta,\Delta^\prime$ by setting $\Delta\leq\Delta^\prime$ if $\b(\Delta)\leq \b(\Delta^\prime)$, $\e(\Delta)\leq \e(\Delta^\prime)$. Note that this is just the product order on $\Z^2$ under the above identifications, or equivalently $\Delta\leq\Delta^\prime$ if and only if $\Delta$ is northwest of $\Delta^\prime$.

		\begin{defi}
			Two segments $\Delta,\Delta^\prime$ are \textit{linked} if $\Delta\cup\Delta^\prime$ is also a segment but neither $\Delta\subseteq\Delta^\prime$ nor $\Delta^\prime\subseteq\Delta$.
			
			If $\Delta,\Delta^\prime$ are linked, then also $\Delta\cup\Delta^\prime$ is a segment and $\Delta\cap\Delta^\prime$ is either the empty set or also a segment. Replacing a pair of linked segments $\Delta+\Delta^\prime$ by $\Delta\cup\Delta^\prime+\Delta\cap\Delta^\prime$ in any multisegment $\m$ is called \textit{union-intersection}. Note that $\Delta\cap\Delta^\prime=0$ can occur for linked segments if $\Delta\leq\Delta^\prime$ and $\e(\Delta)+1=\b(\Delta^\prime)$ (or vice-versa) in which case we say that the segments are \textit{juxtaposed}. When applying union-intersection to such a pair, the empty intersection is discarded and so $\Delta+\Delta^\prime$ is replaced simply by $\Delta\cup\Delta^\prime$.
		\end{defi}
		
		\begin{rmk}
			Identifying segments with points in $\Z^2$, the points $\Delta=(c,d),\Delta^\prime=(c^\prime,d^\prime)$ are linked if and only if one is SouthEast of the other and $(c,d^\prime),(c^\prime,d)\in\{(x,y)\in\Z^2\colon x-1\leq y\}$. The operation of switching coordinates, i.e. $(c,d)+(c^\prime,d^\prime)\mapsto (c,d^\prime)+(c^\prime,d)$ is exactly the union-intersection described above. If $(c,d),(c^\prime,d^\prime)$ correspond to juxtaposed segments with $c<d^\prime$, then $(c^\prime,d)=(d+1,d)\notin T$ corresponds to their intersection. In the present work, we identify all such points with the empty segment and do not allow them in multisegments. In \cite{GL21}, these points are kept as \textit{dummy} segments to show that RSK-standard modules interpolate between Zelevinsky and Langlands standard modules.
		\end{rmk}
		
		The operation of union-intersection is fundamental to many decomposition problems in the representation theory of $\GL_n(F)$ (see \cite{Zelevinsky80} for many instances of this) and will also play a central role in the following.
		
		Finally, we fix some convenient notation which we make frequent use of. For a non-zero segment $\Delta=[c,d]$ we set
		\begin{align*}
			^-\Delta&=[c+1,d],\text{ if }c<d \\
			^+\Delta&=[c-1,d]
		\end{align*}
		If $\Delta=[c,c]$ is a singleton, we set ${}^-\Delta=0$.
		
		Given any multisegment $\m=\sum_{i\in I}\Delta_i$ we set $\min\m\coloneqq \min_{i\in I}\b(\Delta_i)$ and $\max\m\coloneqq\max_{i\in I}\e(\Delta_i)$. Furthermore, if $P$ is any property of segments, we set
		
		\begin{align*}
			\m_P\coloneqq\sum_{\substack{i\in I\\ \Delta_i\text{ satisfies }P}}\Delta_i.
		\end{align*}
	Of primary interest in this article are multisegments of a particularly simple form.
	
	\begin{defi}
		A multisegment $\ml=\Delta_1+\dots+\Delta_k$ is called a \textit{ladder} if the segments can be indexed in such a way that $\b(\Delta_1)>\dots>\b(\Delta_k)$ and $\e(\Delta_1)>\dots>\e(\Delta_k)$, that is $\Delta_j$ is NorthWest of $\Delta_i$ if $i<j$. If the segments are indexed in this way, we refer to $\Delta_1,\dots,\Delta_k$ as the aligned form of $\ml$. In the following, all ladders will be indexed in this way.
	\end{defi}
	
		\subsection{Representation Theory}
		
			After introducing the combinatorial objects of interest, we give a brief account of the way in which they appear in the representation theory of $\mathrm{GL}_n(F)$. We will not introduce the more technical operations needed for this but rather refer the interested reader to \cite{BZ77,Zelevinsky80} for details. Alternatively, these can safely be viewed as a black-box for the purposes of this article.
			
			Let $F$ be any non-Archimedean local field with absolute value $|\cdot|_F$, which for concreteness one can always think of as $\Q_p$, the field of $p$-adic numbers, for some prime $p$ in the following. We recall the classification of smooth complex irreducible representations by multisegments following \cite{BZ77,Zelevinsky80}. Fix an irreducible cuspidal representation $\rho$ of some $\GL_m(F)$. Then a segment $\Delta=[c,d]$ defines an irreducible representation
			\begin{align*}
				Z(\Delta_\rho)\coloneqq\soc(\rho|\cdot|^c\times\dots\times\rho|\cdot|^d),
			\end{align*}
			where $|\cdot|=|\cdot|_F\circ\det$ is a character of $\GL_m(F)$, $\times$ denotes normalised parabolic induction and $\soc(\pi)$ denotes the socle of a representation $\pi$ (i.e. the sum of all irreducible subrepresentations of $\pi$).
			
			If $\m=\sum_{i=1}^k\Delta_i$ is a multisegment, where the indexing satisfies $\Delta_i\nleq\Delta_j$ whenever $i<j$ and $\Delta_i\neq\Delta_j$ (e.g. by taking the opposite of the lexicographic order on segments), then
			\begin{align*}
				Z(\m_\rho)\coloneqq\soc(Z((\Delta_1)_\rho)\times\dots\times Z((\Delta_k)_\rho))
			\end{align*}
			is an irreducible representation and all irreducible representations of any $\GL_n(F)$ with cuspidal support in $\Z\cdot\rho\coloneqq\{\rho|\cdot|^x\colon x\in \Z\}$ are of this form. If $\Z\cdot\rho\neq\Z\cdot\rho^\prime$ for $\rho,\rho^\prime$ cuspidal and irreducible then we call $\rho,\rho^\prime$ \textit{inequivalent}. For inequivalent $\rho,\rho^\prime$ the representation $Z(\m_\rho)\times Z(\m^\prime_{\rho^\prime})$ is also irreducible for all multisegments $\m,\m^\prime$. Now fix a representative for each equivalence class of cuspidals. 
			\begin{thm}[\cite{BZ77},\cite{Zelevinsky80}]
				If $\pi$ is any irreducible representation of $\GL_n(F)$ for any $n$, then there is a unique set of inequivalent representatives of irreducible cuspidal representations $\rho_1,\dots,\rho_k$ and non-zero multisegments $\m_1,\dots,\m_k$ such that
				\begin{align*}
					\pi\cong Z((\m_1)_{\rho_1})\times\dots\times Z((\m_k)_{\rho_k}).
				\end{align*}
			\end{thm}
			In the way we stated it, the Bernstein-Zelevinsky classification has two important consequences. Firstly, any irreducible representation of $\GL_n(F)$ can be described using cuspidal representations and purely combinatorial multisegments. Secondly, questions concerning irreducibility and decomposition problems don't depend on the underlying cuspidal representation and one can always restrict to a single fixed cuspidal $\rho$. Thus one is led to the purely combinatorial way we introduced multisegments in the previous section.
			
			From now on we will fix a cuspidal $\rho$ and drop it from our notation, preferring to write simply $Z(\Delta),Z(\m)$ for the corresponding representations. As outlined above, this is not a serious restriction and for the purposes of this article one can safely think of all irreducible representations of $\GL_n(F)$ as being of the form $Z(\m)$ for a unique multisegment $\m$. 
			
			In general, it is difficult to determine whether for two multisegments $\m,\m^\prime$ the representation $Z(\m)\times Z(\m^\prime)$ is irreducible; even computing the socle $\soc(Z(\m)\times Z(\m^\prime))$ is delicate. However, if $\m$ is any multisegment and $\ml$ is a ladder, then $\soc(Z(\m)\times Z(\ml))$ is irreducible (see \cite{LM16} Proposition 6.15.) and so there is a unique multisegment $\soc(\m,\ml)$ satisfying
			\begin{align*}
				\soc(Z(\m)\times Z(\ml))=Z(\soc(\m,\ml)).
			\end{align*}
			If one of $\m,\ml$ is zero, the socle map is trivial:
			\begin{align*}
				\soc(\m,0)&=\m,\\
				\soc(0,\ml)&=\ml.
			\end{align*}			 
			If $\m,\ml$ are non-zero, the map $(\m,\ml)\mapsto\soc(\m,\ml)$ satisfies a purely combinatorial recursion. The case $\min \ml\leq\min\m$ is rather simple, while the case $\min\ml>\min\m$ is formulated in terms of the first step of the M\oe glin-Waldspurger involution, which is denoted as $\mathcal{MW}(\m)=(\m^\dag,\Delta^\circ(\m))$ and recalled in Section~\ref{sec:mw}.
			\begin{thm}[see \cite{LM16} Proposition 6.15, \cite{GL21} Proposition 4.2]\label{thm:lmrec}
		Let $\m$ be any non-zero multisegment and $\ml$ a non-zero ladder. Then, 
			\begin{enumerate}[a)]
				\item if $\min \ml\leq\min\m$ we have 
					\begin{align*}
						\soc(\m,\ml)=\soc(\m-\m_{\leq\Delta},\ml-\Delta)+\m_{\leq\Delta}+\Delta,
					\end{align*}
					where $\Delta$ is the segment of $\ml$ with $b(\Delta)=\min\ml$, and
				\item if $\min\ml>\min\m$ we have
					\begin{align*}
						\mathcal{MW}(\soc(\m,\ml))=(\soc(\m^\dag,\ml),\Delta^\circ(\m)).
					\end{align*}
					In other words, in this case we have:
					\begin{align*}
						\soc(\m,\ml)^\dag&=\soc(\m^\dag,\ml)\quad\text{and}\\
						\Delta^\circ(\soc(\m,\ml))&=\Delta^\circ(\m).
					\end{align*}
			\end{enumerate}
		\end{thm}
		\subsection{MW-Algorithm}\label{sec:mw}
		
			We recall the combinatorial description of the first step of the M\oe glin-Waldspurger algorithm for the Zelevinsky involution, see \cite{MWInvol}. We follow the orientation conventions of \cite[Section 2.2]{GL21}.

			For this let $\m=\sum_{i\in I}\Delta_i$ be a non-zero multisegment. Let $i_0\in I$ be such that $\b(\Delta_{i_0})=\min\m$ with $\e(\Delta_{i_0})$ minimal. After finding $i_0,\dots, i_{l-1}\in I$, choose, if it exists, $i_l\in I$ such that $\b(\Delta_{i_l})=\b(\Delta_{i_{l-1}})+1$, $\e(\Delta_{i_l})>\e(\Delta_{i_{l-1}})$ and $\e(\Delta_{i_l})$ is minimal for such segments. If no such segment exists, the algorithm terminates.
			
In this way we obtain a sequence of indices $i_0,\dots,i_k\in I$. Note that while this sequence is not necessarily unique, the sequence of segments $\Delta_{i_0},\dots,\Delta_{i_k}$ is and we will call this the \textit{initial sequence} of $\m$.

			We then define a new multisegment $\m^\dag$ by replacing each $\Delta_{i_j}$ with $^-\Delta_{i_j}$ leaving everything else unchanged, i.e. by moving each point in the sequence constructed above down by $1$, removing it if it becomes the empty segment. We denote by $\Delta_i^\dag$ the segment of $\m^\dag$ corresponding to the segment $\Delta_i$ of $\m$ in this way. We also keep track of the rows affected by this by setting $\Delta^\circ(\m)=[\b(\Delta_{i_0}),\b(\Delta_{i_k})]=[\min\m,\min\m+k]$.
			
Together this defines a map $\mathcal{MW}:\m\mapsto (\m^\dag,\Delta^\circ(\m))$ taking a non-zero multisegment and producing a pair of a multisegment and a segment. 
	
	Although this is not immediately obvious from its combinatorial description, the map $\mathcal{MW}$ is injective. This can be seen by using the properties of the full M\oe glin-Waldspurger involution, see for example \cite{GL21} for a proof of this.

\begin{ex}\label{ex:mw}
	Consider the multisegment $\m=[7,7]+[6,8]+[5,5]+[4,6]+[3,6]+[3,4]+[2,5]+[2,3]+[1,6]+[1,1]$, where we index segments in the order they appear in the sum. Identifying $\m$ with a constellation of points in the plane we obtain, the diagram in Figure \ref{fig:mwex1}.
	
			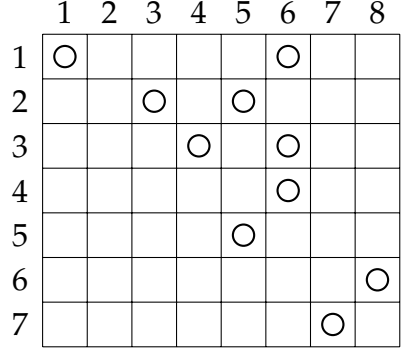
\begin{figure}[ht]
				\centering
				\begin{tikzpicture}[x=1.4em,y=1.4em, thick,color = black]
			    \draw[step=1,black,thin] (0,0) grid (8,7);
			    \draw [black] (0.5,6.5) circle (4pt);
			    \draw [black] (5.5,6.5) circle (4pt);
			    \draw [black] (4.5,5.5) circle (4pt);
    				\draw [black] (2.5,5.5) circle (4pt);
    				\draw [black] (3.5,4.5) circle (4pt);
    				\draw [black] (5.5,4.5) circle (4pt);
    				\draw [black] (5.5,3.5) circle (4pt);
    				\draw [black] (4.5,2.5) circle (4pt);
    				\draw [black] (7.5,1.5) circle (4pt);
    				\draw [black] (6.5,0.5) circle (4pt);
    				\node at (0.5,7.5) {$1$};
    				\node at (1.5,7.5) {$2$};
    				\node at (2.5,7.5) {$3$};
    				\node at (3.5,7.5) {$4$};
    				\node at (4.5,7.5) {$5$};
    				\node at (5.5,7.5) {$6$};
    				\node at (6.5,7.5) {$7$};
    				\node at (7.5,7.5) {$8$};
    				\node at (-0.5,6.5) {$1$};
    				\node at (-0.5,5.5) {$2$};
    				\node at (-0.5,4.5) {$3$};
    				\node at (-0.5,3.5) {$4$};
    				\node at (-0.5,2.5) {$5$};
    				\node at (-0.5,1.5) {$6$};
    				\node at (-0.5,0.5) {$7$};
    				\end{tikzpicture}
    				\caption{A multisegment $\m$ as points in the plane.}
    				\label{fig:mwex1}
			\end{figure}
			
	To compute $\mathcal{MW}(\m)$ we first find $	\Delta_{i_0}=[1,1]$ since $\b([1,1])=\min\m$ and $\e([1,1])<\e([1,6])$. Then, in the next row, the first point that is SE of $(1,1)$ is $(2,3)$ so $\Delta_{i_1}=[2,3]$. Continuing, we find that the initial sequence of $\m$ is given by $[1,1],[2,3],[3,4],[4,6]$, which we indicate in red in Figure \ref{fig:exinit}.
	
			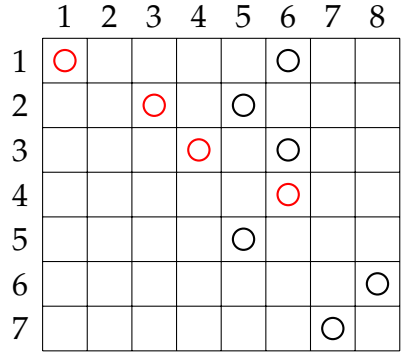
\begin{figure}[ht]
				\centering
				\begin{tikzpicture}[x=1.4em,y=1.4em, thick,color = black]
			    \draw[step=1,black,thin] (0,0) grid (8,7);
			    \draw [red] (0.5,6.5) circle (4pt);
			    \draw [black] (5.5,6.5) circle (4pt);
			    \draw [black] (4.5,5.5) circle (4pt);
    				\draw [red] (2.5,5.5) circle (4pt);
    				\draw [red] (3.5,4.5) circle (4pt);
    				\draw [black] (5.5,4.5) circle (4pt);
    				\draw [red] (5.5,3.5) circle (4pt);
    				\draw [black] (4.5,2.5) circle (4pt);
    				\draw [black] (7.5,1.5) circle (4pt);
    				\draw [black] (6.5,0.5) circle (4pt);
    				\node at (0.5,7.5) {$1$};
    				\node at (1.5,7.5) {$2$};
    				\node at (2.5,7.5) {$3$};
    				\node at (3.5,7.5) {$4$};
    				\node at (4.5,7.5) {$5$};
    				\node at (5.5,7.5) {$6$};
    				\node at (6.5,7.5) {$7$};
    				\node at (7.5,7.5) {$8$};
    				\node at (-0.5,6.5) {$1$};
    				\node at (-0.5,5.5) {$2$};
    				\node at (-0.5,4.5) {$3$};
    				\node at (-0.5,3.5) {$4$};
    				\node at (-0.5,2.5) {$5$};
    				\node at (-0.5,1.5) {$6$};
    				\node at (-0.5,0.5) {$7$};
    				\end{tikzpicture}
    				\caption{The initial sequence of $\m$.}
    				\label{fig:exinit}
			\end{figure}
			
	If $\mathcal{MW}(\m)=(\m^\dag,\Delta^\circ(\m))$ then $\m^\dag$ is obtained by moving all red points down one row and leaving everything else unchanged, see Figure \ref{fig:mwex2}. Note that the northwesternmost point $(1,1)$ disappears since it becomes the empty segment $0$. Thus we obtain $\m^\dag=[7,7]+[6,8]+[5,6]+[5,5]+[4,4]+[3,6]+[3,3]+[2,5]+[1,6]$. The segment $\Delta^\circ(\m)=[1,4]$ keeps track of the rows that contain a red dot in $\m$.
	
			\begin{figure}[ht]
				\centering
				\begin{tikzpicture}[x=1.4em,y=1.4em, thick,color = black]
			    \draw[step=1,black,thin] (0,0) grid (8,7);
			    \draw [black] (5.5,6.5) circle (4pt);
			    \draw [black] (4.5,5.5) circle (4pt);
    				\draw [black] (2.5,4.5) circle (4pt);
    				\draw [black] (3.5,3.5) circle (4pt);
    				\draw [black] (5.5,4.5) circle (4pt);
    				\draw [black] (5.5,2.5) circle (4pt);
    				\draw [black] (4.5,2.5) circle (4pt);
    				\draw [black] (7.5,1.5) circle (4pt);
    				\draw [black] (6.5,0.5) circle (4pt);
    				\node at (0.5,7.5) {$1$};
    				\node at (1.5,7.5) {$2$};
    				\node at (2.5,7.5) {$3$};
    				\node at (3.5,7.5) {$4$};
    				\node at (4.5,7.5) {$5$};
    				\node at (5.5,7.5) {$6$};
    				\node at (6.5,7.5) {$7$};
    				\node at (7.5,7.5) {$8$};
    				\node at (-0.5,6.5) {$1$};
    				\node at (-0.5,5.5) {$2$};
    				\node at (-0.5,4.5) {$3$};
    				\node at (-0.5,3.5) {$4$};
    				\node at (-0.5,2.5) {$5$};
    				\node at (-0.5,1.5) {$6$};
    				\node at (-0.5,0.5) {$7$};
    				\end{tikzpicture}
    				\caption{The multisegment $\m^\dag$.}
    				\label{fig:mwex2}
			\end{figure}
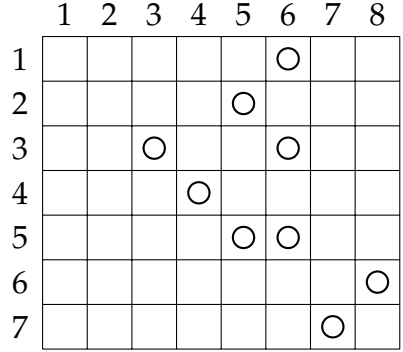
\end{ex}
	The following simple lemma will be useful in Section \ref{sec:mainthm}.
	\begin{lem}\label{lem:singleton}
		Let $\m=\sum_{i\in I} \Delta_i$ be a multisegment with initial sequence $\Delta_{i_0},\dots,\Delta_{i_l}$. If, for some $0\leq k\leq l$, $\Delta_{i_k}$ is a singleton, then all $\Delta_{i_j}$ with $j\leq k$ are singletons. In particular, if $\Delta_{i_0}$ is not a singleton, then there are no singletons in the initial sequence.
	\end{lem}
	\begin{proof}
		If $k=0$, there is nothing to prove. Otherwise, it is enough to show that $\Delta_{i_{k-1}}$ is a singleton, the full claim then follows by induction. By the properties defining an initial sequence and the fact that $\Delta_{i_k}$ is a singleton, we have $\b(\Delta_{i_{k-1}})=\b(\Delta_{i_k})-1=\e(\Delta_{i_k})-1$ and $\b(\Delta_{i_{k-1}})\leq\e(\Delta_{i_{k-1}})<\e(\Delta_{i_k})=\b(\Delta_{i_{k-1}})+1$. But this means that $\b(\Delta_{i_{k-1}})=\e(\Delta_{i_{k-1}})$, so $\Delta_{i_{k-1}}$ is also a singleton.
	\end{proof}
	
	\subsection{Pairs of Ladders}
	
		To give a non-recursive combinatorial description of $\soc(Z(\ml_1)\times Z(\ml_2))$ for ladders $\ml_1,\ml_2$ we need to lay some groundwork regarding pairs $(\ml_1,\ml_2)$ of ladders. The main concept we introduce here is that of permissibility, which is the technical condition we need to make our combinatorial description work (see Lemma \ref{lem:perm} for the precise statement).
		\begin{defi}\label{def:dominant}
			A pair $(\ml_1,\ml_2)$ of ladders with aligned forms $\ml_1=\Delta_1+\dots+\Delta_k$, $\ml_2=\Delta_1^\prime+\dots+\Delta_{k^\prime}^\prime$ is called \textit{dominant}, if $k\leq k^\prime$ and $\Delta_i\leq\Delta_i^{\prime}$ for $i=1,\dots,k$.
		\end{defi}
		Note that for any ladder $\ml$, the pair $(0,\ml)$ is dominant, since the inequalities between segments in the definition are vacuously true.
		\begin{ex}
			The pair of ladders $([3,4]+[1,2],[4,5]+[2,4])$ is dominant, while the pair $([3,4]+[2,3]+[1,1],[4,5]+[2,4])$ is not. In Figure \ref{fig:domex} these examples are illustrated in $\Z^2$ with segments of the first ladder drawn in white and those of the second ladder drawn in black.
			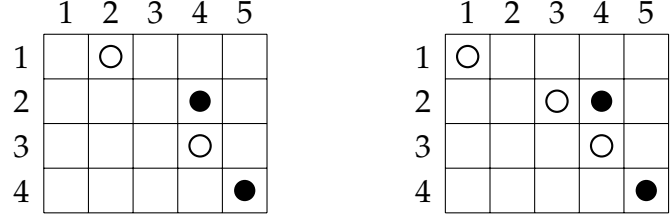
\begin{figure}
				\centering
				\begin{tikzpicture}[x=1.4em,y=1.4em, thick,color = black]
					\draw[step=1,black,thin] (0,0) grid (5,4);
					\draw [black] (1.5,3.5) circle (4pt);
    					\fill [black] (3.5,2.5) circle (4pt);
    					\draw [black] (3.5,1.5) circle (4pt);
    					\fill [black] (4.5,0.5) circle (4pt);
    					\node at (-0.5,3.5) {$1$};
    					\node at (-0.5,2.5) {$2$};
    					\node at (-0.5,1.5) {$3$};
    					\node at (-0.5,0.5) {$4$};
    					\node at (0.5,4.5) {$1$};
    					\node at (1.5,4.5) {$2$};
    					\node at (2.5,4.5) {$3$};
    					\node at (3.5,4.5) {$4$};
    					\node at (4.5,4.5) {$5$};
    				\end{tikzpicture} \qquad \qquad
    				\begin{tikzpicture}[x=1.4em,y=1.4em, thick,color = black]
					\draw[step=1,black,thin] (0,0) grid (5,4);
					\draw [black] (0.5,3.5) circle (4pt);
					\draw [black] (2.5,2.5) circle (4pt);
    					\fill [black] (3.5,2.5) circle (4pt);
    					\draw [black] (3.5,1.5) circle (4pt);
    					\fill [black] (4.5,0.5) circle (4pt);
    					\node at (-0.5,3.5) {$1$};
    					\node at (-0.5,2.5) {$2$};
    					\node at (-0.5,1.5) {$3$};
    					\node at (-0.5,0.5) {$4$};
    					\node at (0.5,4.5) {$1$};
    					\node at (1.5,4.5) {$2$};
    					\node at (2.5,4.5) {$3$};
    					\node at (3.5,4.5) {$4$};
    					\node at (4.5,4.5) {$5$};
    				\end{tikzpicture}
    				\caption{An example and a non-example of a pair of ladders being dominant}
    				\label{fig:domex}
			\end{figure}
		\end{ex}
		\begin{lem}\label{lem:dagpair}
			Let $(\ml_1,\ml_2)$ be a dominant pair of non-zero ladders with $\min\ml_1<\min\ml_2$. Then also $(\ml_1^\dag,\ml_2)$ is dominant.
		\end{lem}
		\begin{proof}
			Index $\ml_1,\ml_2$ with aligned forms $\ml_1=\Delta_1+\dots+\Delta_k$, $\ml_2=\Delta_1^\prime+\dots+\Delta_{k^\prime}^\prime$ and let $\Delta_k,\dots,\Delta_{k-l}$ be the initial sequence of $\ml_1$ for some $l\in\{0,1,\dots,k-1\}$. If $\ml_1$ has singletons in its initial sequence, these disappear in $\ml_1^\dag$. By Lemma \ref{lem:singleton}, all singletons in the initial sequence are $\Delta_k,\Delta_{k-1},\dots,\Delta_s$ for some $s$ (possibly $s=k+1$ if there are no such segments). When checking dominance of $(\ml_1^\dag,\ml_2)$, we thus only need to consider segments $\Delta_i^\dag,\Delta_i^\prime$ with $i<s$. It is enough to show $\b(\Delta_i)<\b(\Delta_i^\prime)$ for all $i=k,k-1,\dots,k-l$, since then $\Delta_i^\dag\leq\Delta_i^\prime$ for all $i=1,2,\dots,s-1$ so that $(\ml_1^\dag,\ml_2)$ is dominant.
			
			 This can be seen by induction on $i$: for $i=k$ we have by assumption $\b(\Delta_k)=\min\ml_1<\min\ml_2\leq \b(\Delta_k^\prime)$. Now assume the claim holds for some $i=k,k-1,\dots,k-l+1$, then we conclude $\b(\Delta_{i-1})=\b(\Delta_i)+1<\b(\Delta_i^\prime)+1\leq \b(\Delta_{i-1}^\prime)$ by induction and the fact that $\ml_2$ is a ladder.
		\end{proof}
		We will also need a partial converse to the previous lemma.
		\begin{lem}\label{lem:dagpairrev}
			Let $\ml_1,\ml_2$ be non-zero ladders such that $\ml_1$ has no singletons in its initial sequence and $(\ml_1^\dag,\ml_2)$ is dominant. Then $(\ml_1,\ml_2)$ is also dominant.
		\end{lem}
		\begin{proof}
			The assumption on $\ml_1$ ensures that $\ml_1^\dag$ and $\ml_1$ have the same number of segments and the dominance of $(\ml_1^\dag,\ml_2)$ then shows that $\ml_1$ has at most as many segments as $\ml_2$. For any two segments $\Delta,\Delta^\prime$ where $\Delta$ is not a singleton, if ${}^-\Delta\leq\Delta^\prime$ then clearly also $\Delta\leq\Delta^\prime$. These two facts together immediately show the dominance of $(\ml_1,\ml_2)$.
		\end{proof}
		\begin{rmk}
			In \cite{GL21} the RSK-correspondence is used to construct tuples of ladders $(\ml_1,\dots,\ml_k)$ with prescribed $\soc(Z(\ml_k)\times\dots\times Z(\ml_1))$. The monotonicity conditions in columns in the tableaux appearing in the RSK-correspondence translate to the above dominance condition on pairs of ladders. In the following we adapt these ideas to cover also the case of non-dominant $(\ml_1,\ml_2)$, see also Section \ref{sec:appl} where the connection to RSK is explored in more detail.
		\end{rmk}
		
	To deal also with non-dominant pairs $(\ml_1,\ml_2)$ we will decompose such pairs into a part that is dominant and one that is not. We do this in such a way that the dominant part is as large as possible (in a certain sense), which automatically gives the remaining pair a property that is dual to dominance. Then we can treat the two pieces in a very similar fashion.
	
	\begin{defi}
		Let $(\ml_1,\ml_2)$ be a non-dominant pair of ladders. Then we define $\Gamma(\ml_1,\ml_2)$ to be the southeasternmost segment $\Delta$ of $\ml_1$ such that $((\ml_1)_{\ngeq\Delta},(\ml_2)_{\ngeq\Delta})$ is dominant.
		
		If $(\ml_1,\ml_2)$ is dominant, we set $\Gamma(\ml_1,\ml_2)=[\max\ml_2+1,\max\ml_2+1]$ to ensure that $\Delta\leq \Gamma(\ml_1,\ml_2)$ for all segments $\Delta$ of $\ml_1+\ml_2$.
		
		If the ladders are clear from context, we simply write $\Gamma$ for $\Gamma(\ml_1,\ml_2)$.
	\end{defi}
		For a segment $\Delta$ and a multisegment $\m$, the multisegment $\m_{\ngeq\Delta}$ is simply the part of $\m$ consisting of all segments that are not southeast of $\Delta$ under our standard identifications. We will use the segment $\Gamma(\ml_1,\ml_2)$ to split the pair $(\ml_1,\ml_2)$ into the part southeast of $\Gamma$ and the remaining part.
		\begin{ex}\label{ex:Psi}
			We illustrate the previous definition in an example. For this let $\ml_1=[6,16]+[5,14]+[4,13]+[2,11]+[1,9], \ml_2=[9,17]+[8,15]+[7,12]+[3,10]$ be two ladders. We will use this as a running example throughout the article. We will always draw pairs of ladders $(\ml_1,\ml_2)$ by colouring the segments of $\ml_1$ in white and those of $\ml_2$ in black. See also Figure \ref{fig:pairlad}.
			
			Going through the segments of $\ml_1$ and checking the condition on $\Gamma(\ml_1,\ml_2)$ one finds $\Gamma(\ml_1,\ml_2)=[4,13]$. In Figure \ref{fig:pairlad} this is emphasised by colouring the segment $\Gamma=\Gamma(\ml_1,\ml_2)$ in orange as well as indicating the region that is southeast resp. not southeast of it by an orange line.
			\begin{figure}[ht]
			\centering
			\begin{tikzpicture}[x=1.4em,y=1.4em, thick,color = black]
				\draw[step=1,black,thin] (0,0) grid (9,9);
				\draw [black] (0.5,8.5) circle (4pt);
				\draw [black] (2.5,7.5) circle (4pt);
    				\fill [orange] (4.5,5.5) circle (4pt);
    				\draw [black] (5.5,4.5) circle (4pt);
    				\draw [black] (7.5,3.5) circle (4pt);
    				\fill (1.5,6.5) circle (4pt);
    				\fill (3.5,2.5) circle (4pt);
    				\fill (6.5,1.5) circle (4pt);
    				\fill (8.5,0.5) circle (4pt);
    				\draw [orange] (4,-0.5) -- (4,6) -- (9.5,6);
    				\node at (0.5,9.5) {$9$};
    				\node at (1.5,9.5) {$10$};
    				\node at (2.5,9.5) {$11$};
    				\node at (3.5,9.5) {$12$};
    				\node at (4.5,9.5) {$13$};
    				\node at (5.5,9.5) {$14$};
    				\node at (6.5,9.5) {$15$};
    				\node at (7.5,9.5) {$16$};
    				\node at (8.5,9.5) {$17$};
    				\node at (-0.5,8.5) {$1$};
    				\node at (-0.5,7.5) {$2$};
    				\node at (-0.5,6.5) {$3$};
    				\node at (-0.5,5.5) {$4$};
    				\node at (-0.5,4.5) {$5$};
    				\node at (-0.5,3.5) {$6$};
    				\node at (-0.5,2.5) {$7$};
    				\node at (-0.5,1.5) {$8$};
    				\node at (-0.5,0.5) {$9$};
			\end{tikzpicture}
			\caption{A pair of ladders in the plane with cut-off.}
			\label{fig:pairlad}
		\end{figure}
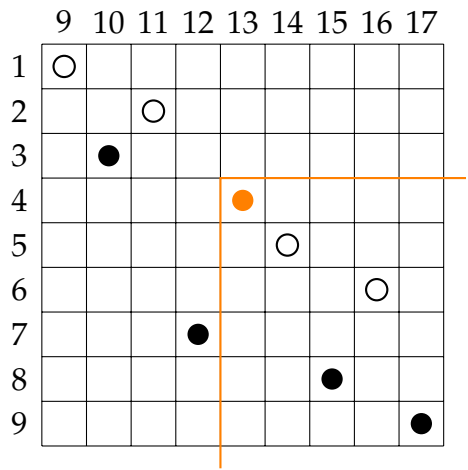
		Looking at this example, one may notice that $((\ml_1)_{\geq\Gamma},(\ml_2)_{\geq\Gamma})$, i.e. the white and black points southeast of the orange line satisfy a dual version of the dominance condition: writing $(\ml_1)_{\geq\Gamma}=\Delta_1+\dots+\Delta_l+\Gamma$ and $(\ml_2)_{\geq\Gamma}=\Delta_1^\prime+\dots+\Delta_{l^\prime}^\prime$ in aligned form, we have $l\geq l^\prime$ and for each $i=0,\dots,l^\prime-1$ we have $\Delta_{l-i}\leq\Delta_{l^\prime-i}^\prime$. This is not a coincidence as the following lemma shows.
		\end{ex}
		\begin{lem}\label{lem:dualdom}
			Let $(\ml_1,\ml_2)$ be a non-dominant pair of ladders. Writing $(\ml_1)_{\geq\Gamma}=\Delta_1+\dots+\Delta_l+\Gamma$, $(\ml_2)_{\geq\Gamma}=\Delta_1^\prime+\dots+\Delta_{l^\prime}^\prime$ in aligned form, we have $l\geq l^\prime$ and $\Delta_{l-i}\leq\Delta_{l^\prime-i}^\prime$ for $i=0,\dots,l^\prime-1$.
		\end{lem}
		\begin{proof}
			We argue by contradiction. So assume that there is some $i=0,\dots,l^\prime-1$ such that $\Delta_{l-i}\nleq\Delta_{l^\prime-i}^\prime$ or such that $\Delta_{l-i}$ does not exist, i.e. such that $i\geq l$, and let $j$ be the minimal such index. We will derive a contradiction to the maximality in the definition of $\Gamma$. If $j=l$, set $\Delta_0=[\max\ml_2+1,\max\ml_2+1]$ so that no segment $\Delta^\prime$ of $\ml_2$ satisfies $\Delta_0\leq\Delta^\prime$.
			
			By assumption we have $\Delta_{l-j}\nleq\Delta_{l^\prime-j}^\prime$ and $\Delta_{l-i}\leq\Delta_{l^\prime-i}^\prime$ for $i=0,\dots,j-1$. Since we have $\Gamma\leq \Delta_{l^\prime-i}^\prime$ for all $i=0,\dots,j$ by construction, this shows that $(\Delta_{l-j+1}+\dots+\Delta_l+\Gamma,\Delta_{l^\prime-j}^\prime+\dots+\Delta_{l^\prime}^\prime)$ is dominant. But then also $((\ml_1)_{\ngeq\Delta_{l-j}},(\ml_2)_{\ngeq\Delta_{l-j}})$ is dominant. If $j<l$, this contradicts the fact that $\Gamma$ is the $\leq$-maximal segment of $\ml_1$ with this property and $\Gamma<\Delta_{l-j}$. If $j=l$ then $((\ml_1)_{\ngeq\Delta_{0}},(\ml_2)_{\ngeq\Delta_{0}})=(\ml_1,\ml_2)$ contradicting the assumption that the pair is not dominant.
		\end{proof}
		\begin{rmk}
			The map $\Delta\mapsto\Delta^\vee=[-\e(\Delta),-\b(\Delta)]$ on segments induces the map $Z(\m)\mapsto Z(\m^\vee)$ on irreducible representations by applying it to each segment of $\m$. It is well-known that $Z(\m^\vee)=Z(\m)^\vee$, the contragredient of $Z(\m)$. Strictly speaking, one also has to change the indexing cuspidal from $\rho$ to $\rho^\vee$ when taking contragredients, but we omit this in the notation since it does not affect the arguments in the present work. Lemma \ref{lem:dualdom} can thus also be rephrased to say that $(((\ml_2)_{\geq\Gamma})^\vee,((\ml_1)_{\geq\Gamma})^\vee)$ is dominant and in this sense $((\ml_1)_{\geq\Gamma},(\ml_2)_{\geq\Gamma})$ satisfies a condition that is dual to dominance.
		\end{rmk}
		
		Roughly, we will give a way of computing $\soc(\ml_1,\ml_2)$ that works by pairing certain segments $\Delta$ of $\ml_1$ with segments $\Delta^\prime$ of $\ml_2$ and doing union-intersection on these pairs. To make sure that this is well-defined, i.e. that union-intersection actually produces segments, we need to impose a condition on the pair $(\ml_1,\ml_2)$ of ladders.
		\begin{defi}\label{def:perm}
			Let $(\ml_1,\ml_2)$ be a pair of ladders with aligned forms $\ml_1=\Delta_1+\dots+\Delta_k$, $\ml_2=\Delta_1^\prime+\dots+\Delta_{k^\prime}^\prime$.
			\begin{enumerate}[a)]
				\item If $(\ml_1,\ml_2)$ is dominant, we call it \textit{permissible} if for all $i=1,\dots,k$ and $j$ maximal such that $\Delta_i\leq \Delta_j^\prime$ we have that $\e(\Delta_r)+1\geq\b(\Delta_{j-i+r}^\prime)$ for $r=1,\dots,i$.
				\item If $(\ml_1,\ml_2)$ is not dominant we call it \textit{permissible} if, with $\Gamma=\Gamma(\ml_1,\ml_2)$, both dominant pairs $((\ml_1)_{\ngeq\Gamma},(\ml_2)_{\ngeq\Gamma})$ and $(((\ml_2)_{\geq\Gamma})^\vee,((\ml_1)_{\geq\Gamma})^\vee)$ are permissible.
			\end{enumerate}
		\end{defi}
		We remark that, since $(0,\ml)$ is dominant for any ladder, one easily sees that $(\ml_1,\ml_2)$ is permissible whenever at least one of $\ml_1,\ml_2$ is zero.
		\begin{ex}
			An example of a permissible pair of ladders is given in Figure \ref{fig:pairlad}. Non-permissible pairs $(\ml_1,\ml_2)$ can only occur if there are segments $\Delta_1,\Delta_2$ in $\ml_1+\ml_2$ such that $\e(\Delta_1)+1<\b(\Delta_2)$, see also the following remark. In Figure \ref{fig:nonperm} the non-permissible pair $(\ml_1,\ml_2)=([5,6]+[4,4]+[1,1],[6,6]+[3,4])$ is depicted in $\Z^2$. The pair is non-permissible since $\Gamma(\ml_1,\ml_2)=[4,4]$ and the northwest pair is then $([1,1],[3,4])$, but $[1,1]\cup[3,4]$ is not a segment. In Lemma \ref{lem:perm} the reason for this definition of permissibility will become clear.
		
		\begin{figure}[ht]
			\centering
			\begin{tikzpicture}
				[x=1.4em,y=1.4em, thick,color = black]
				\draw[step=1,black,thin] (0,0) grid (6,6);
				\draw [black] (0.5,5.5) circle (4pt);
				\draw [black] (3.5,2.5) circle (4pt);
    				\draw [black] (5.5,1.5) circle (4pt);
    				\fill [black] (3.5,3.5) circle (4pt);
    				\fill [black] (5.5,0.5) circle (4pt);
    				\node at (0.5,6.5) {$1$};
    				\node at (1.5,6.5) {$2$};
    				\node at (2.5,6.5) {$3$};
    				\node at (3.5,6.5) {$4$};
    				\node at (4.5,6.5) {$5$};
    				\node at (5.5,6.5) {$6$};
    				\node at (-0.5,5.5) {$1$};
    				\node at (-0.5,4.5) {$2$};
    				\node at (-0.5,3.5) {$3$};
    				\node at (-0.5,2.5) {$4$};
    				\node at (-0.5,1.5) {$5$};
    				\node at (-0.5,0.5) {$6$};
			\end{tikzpicture}
			\caption{A non-permissible pair of ladders.}
			\label{fig:nonperm}
		\end{figure}
		
		\end{ex}
		\begin{rmk}
			Our notion of permissibility extends that of \cite{GL21} in that we allow juxtaposed segments in the dominant case and that we also include non-dominant pairs in general.
		\end{rmk}
		The following is obvious from the definition of permissibility.
		\begin{lem}
			If $(\ml_1,\ml_2)$ is a pair of ladders that satisfies $\b(\Delta)\leq \e(\Delta^\prime)+1$ for all segments $\Delta,\Delta^\prime$ of $\ml_1+\ml_2$ then $(\ml_1,\ml_2)$ is permissible.
		\end{lem}
		Thus our setup encompasses the setting of regular pairs of ladders as in \cite{Gurevich2021} where $\soc(Z(\ml_1)\times Z(\ml_2))$ was described in a somewhat different fashion to our setting in Theorem 6.11 of that paper.
		
		In Section \ref{sec:appl} we will tie the construction of $\soc(\ml_1,\ml_2)$ below in with classical RSK combinatorics. There we restrict to the setting $\b(\Delta)\leq \e(\Delta^\prime)$ and the above lemma guarantees that we do not have to worry about permissibility there.
	\section{Combinatorial Construction}\label{sec:mat}
	
		The goal of this section is to introduce a combinatorial map $\Psi$ mapping permissible pairs $(\ml_1,\ml_2)$ to a multisegment $\Psi(\ml_1,\ml_2)$. In Theorem \ref{thm:soc} we will prove that in fact $\Psi(\ml_1,\ml_2)=\soc(\ml_1,\ml_2)$ so that $\Psi$ computes $\soc(Z(\ml_1)\times Z(\ml_2))$ in a combinatorial way.
		
		Roughly, $\Psi$ is defined by matching segments of $\ml_1$ with segments of $\ml_2$ and performing union-intersection on matched pairs. The permissibility condition we introduced in the previous section will guarantee that this is well-defined, i.e. that all matched pairs $(\Delta,\Delta^\prime)$ of segments satisfy $\Delta\leq\Delta^\prime$ and $\Delta\cup\Delta^\prime$ is also a segment, so that the union-intersection of them is well-defined. 
		
		Let $(\ml_1,\ml_2)$ be any pair of ladders. To simplify notation, we refer to $((\ml_1)_{\geq\Gamma},(\ml_2)_{\geq\Gamma})$ as the \textit{southeast} part and to $((\ml_1)_{\ngeq\Gamma},(\ml_2)_{\ngeq\Gamma})$ as the \textit{northwest} part of $(\ml_1,\ml_2)$. And we think of $(\ml_1,\ml_2)$ as white and black points in the plane as before. It is instructive to look at Figure \ref{fig:pairlad} and check that the result of the procedure explained below is indeed as in Figure \ref{fig:expsi}. If one (or both) of $\ml_1,\ml_2$ is $0$, we set $\Psi(\ml_1,\ml_2)=\ml_1+\ml_2$. Assume now that both ladders are non-zero.
		
		In the northwest part, starting with the north-most white point, we go over every white point and match it with the first black point that is southeast of it and has not been matched before. In the southeast part, we perform the 'dual' matching: starting with the south-most black point, we go over every black point and match it with the first white point that is northwest of it and has not previously been matched.
		
		The construction of $\Gamma$ guarantees that this matching is always possible: the northwest part of $(\ml_1,\ml_2)$ is dominant and it is easy to see that this dominance guarantees that all white points in this region will be matched to a black point in this way. By Lemma \ref{lem:dualdom} the southeast part of $(\ml_1,\ml_2)$ enjoys a 'dual-dominance' condition, guaranteeing that all black points are matched in this part. In fact the lemma also shows that $\Gamma$ itself will not appear as a matched segment in $M(\ml_1,\ml_2)$.
		
		This defines a set $M(\ml_1,\ml_2)=\{(\Delta,\Delta^\prime)\colon \Delta\text{ matched with }\Delta^\prime\}$ of pairs $(\Delta,\Delta^\prime)$ where $\Delta$ is a segment of $\ml_1$ and $\Delta^\prime$ is a segment of $\ml_2$. By construction, we have $\Delta\leq\Delta^\prime$ (i.e. the white point is northwest of the black point) for each  $(\Delta,\Delta^\prime)\in M(\ml_1,\ml_2)$. If $(\ml_1,\ml_2)$ is permissible, we can say even more.
		\begin{lem}\label{lem:perm}
			Let $(\ml_1,\ml_2)$ be a pair of ladders. The following are equivalent.
			\begin{enumerate}[i)]
				\item $(\ml_1,\ml_2)$ is permissible.
				\item $\Delta\cup\Delta^\prime$ is a segment for all $(\Delta,\Delta^\prime)\in M(\ml_1,\ml_2)$.
			\end{enumerate}
		\end{lem}
		\begin{proof}
			Considering the definition of permissibility and Lemma \ref{lem:dualdom}, it is enough to prove this in the case that $(\ml_1,\ml_2)$ is dominant. Then our claim is essentially (with some adjustments to the permissibility condition) contained in the proof of \cite[Proposition 2.4]{GL21} but for clarity we isolate the argument here. 
			
			By the above discussion, dominance implies that all segments of $\ml_1$ are covered by $M\coloneqq M(\ml_1,\ml_2)$. Writing $\ml_1=\Delta_1+\dots+\Delta_k$, $\ml_2=\Delta_1^\prime+\dots+\Delta_{k^\prime}^\prime$ in aligned form, let $p:[k]\rightarrow [k^\prime]$ be the map defined by the matching $M(\ml_1,\ml_2)$ and let $s:[k]\rightarrow [k^\prime]$ be $s(i)=\max\{j\in[k^\prime]\colon \Delta_i\leq\Delta_j^\prime\}$. Then the matching rules explained above show that $p(k)=s(k)$ since the northwest-most segment of $\ml_1$ gets paired to the nw-most segment of $\ml_2$ that is southeast of it. For $i<k$ we have $p(i)=\min\{s(i),p(i+1)-1\}$ corresponding to the cases where $\Delta_{s(i)}^\prime$ is available to be paired with $\Delta_i$ or whether it was already matched before. An easy induction then shows 
			\begin{align*}
				p(i)=\min_{i\leq j\leq k}(s(j)-j+i),
			\end{align*}
		for any $i\in [k]$. Now we prove the equivalence in the statement of the lemma.
		
			Assume first that $(\ml_1,\ml_2)$ is permissible and let $i\in [k]$. From the above, there is $i\leq j\leq k$ such that $p(i)=s(j)-j+i$. The permissibility then implies (take $j$ as $i$ and $i$ as $r$ in the definition of permissibility)
				\begin{align*}
					\e(\Delta_{i})+1\geq \b(\Delta^\prime_{s(j)-(j)+i})=\b(\Delta^\prime_{p(i)}).
				\end{align*}
			Since $\Delta_i\leq\Delta_{p(i)}^\prime$ by construction, this shows that $\Delta_i\cup\Delta_{p(i)}^\prime$ is a segment for all $i=1,\dots,k$, which is what we wanted to show.
			
			Now assume that ii) holds. As above, we note that by construction of $M$, we have 
			\begin{align*}
				p(i)\leq s(i+r)-r
			\end{align*}
			for all $i=1,\dots,k$ and $r=0,\dots,k-i$. Plugging in the definition of $s$ and recalling that $\Delta_i\leq \Delta_{p(i)}^\prime$, permissibility of $(\ml_1,\ml_2)$ follows.
		\end{proof}
		
		The lemma guarantees that the following is well-defined.
		\begin{defi}
			Let $(\ml_1,\ml_2)$ be a permissible pair of ladders and $M(\ml_1,\ml_2)$ as defined above, with $M(\ml_1,\ml_2)=\emptyset$ if $\ml_1$ or $\ml_2$ is empty. Then we define $\Psi(\ml_1,\ml_2)$ to be the multisegment obtained from $\ml_1+\ml_2$ by doing union-intersection on all pairs $(\Delta,\Delta^\prime)\in M(\ml_1,\ml_2)$. In other words
			\begin{align*}
				\Psi(\ml_1,\ml_2)\coloneqq\sum_{(\Delta,\Delta^\prime)\in M(\ml_1,\ml_2)}\lr{\Delta\cup\Delta^\prime+\Delta\cap\Delta^\prime}+\sum_{\substack{\Delta\text{ in }\ml_1\\ \Delta\text{ unmatched}}}\Delta+\sum_{\substack{\Delta^\prime\text{ in }\ml_2\\ \Delta^\prime\text{ unmatched}}}\Delta^\prime.
			\end{align*}
		\end{defi}
		In our running example this produces the multisegment $\Psi(\ml_1,\ml_2)=[9,16]+[8,14]+[7,11]+[6,17]+[5,15]+[4,13]+[3,9]+[2,12]+[1,10]$. Comparing with Figure \ref{fig:pairlad} this produces the picture after 'swapping' matched points and forgetting about the colouring.
		\begin{figure}[ht]
		\centering
		\begin{tikzpicture}[x=1.4em,y=1.4em, thick,color = black]
				\draw[step=1,black,thin] (0,0) grid (9,9);
				\draw [black] (0.5,6.5) circle (4pt);
				\draw [black] (2.5,2.5) circle (4pt);
    				\draw [black] (4.5,5.5) circle (4pt);
    				\draw [black] (5.5,1.5) circle (4pt);
    				\draw [black] (7.5,0.5) circle (4pt);
    				\draw [black] (8.5,3.5) circle (4pt);
    				\draw [black] (1.5,8.5) circle (4pt);
    				\draw [black] (3.5,7.5) circle (4pt);
    				\draw [black] (6.5,4.5) circle (4pt);
    				\node at (0.5,9.5) {$9$};
    				\node at (1.5,9.5) {$10$};
    				\node at (2.5,9.5) {$11$};
    				\node at (3.5,9.5) {$12$};
    				\node at (4.5,9.5) {$13$};
    				\node at (5.5,9.5) {$14$};
    				\node at (6.5,9.5) {$15$};
    				\node at (7.5,9.5) {$16$};
    				\node at (8.5,9.5) {$17$};
    				\node at (-0.5,8.5) {$1$};
    				\node at (-0.5,7.5) {$2$};
    				\node at (-0.5,6.5) {$3$};
    				\node at (-0.5,5.5) {$4$};
    				\node at (-0.5,4.5) {$5$};
    				\node at (-0.5,3.5) {$6$};
    				\node at (-0.5,2.5) {$7$};
    				\node at (-0.5,1.5) {$8$};
    				\node at (-0.5,0.5) {$9$};
			\end{tikzpicture}
		\caption{The result of applying $\Psi$ to $(\ml_1,\ml_2)$ of Figure \ref{fig:pairlad}.}
		\label{fig:expsi}
	\end{figure}
	\begin{rmk}\label{rmk:gl21}
		This defines a map $\Psi$ taking permissible pairs and producing multisegments. If $(\ml_1,\ml_2)$ is dominant it is not hard to see that our construction of $\Psi(\ml_1,\ml_2)$ coincides with the description of $\mathcal{K}^\prime$ in \cite[Proposition 2.4]{GL21} on the common domain of definition. Our main point of departure from that article is dropping the dominance assumption on $(\ml_1,\ml_2)$ and also allowing juxtaposed segments in the permissibility condition on the pairs.
		
		As the previous lemma shows, this construction cannot be used to compute $\soc(\ml_1,\ml_2)$ without the permissibility condition of $(\ml_1,\ml_2)$. It would be interesting to give a similar description, i.e. one that matches segments of $\ml_1$ and of $\ml_2$ and performs one set of union-intersections, of $\soc(\ml_1,\ml_2)$ for a general pair of ladders. 
	\end{rmk}	
		The way we constructed $\Psi$ in each region can be seen as a description of the inverse of Viennot's shadow construction of RSK with different orientations in each region. We explore this connection more thoroughly in Section \ref{sec:appl}.
		
	\section{Main Theorem}\label{sec:mainthm}
		In this section we prove our main result.
		\begin{thm}\label{thm:soc}
			Let $(\ml_1,\ml_2)$ be a permissible pair of ladders. Then $\Psi(\ml_1,\ml_2)=\soc(\ml_1,\ml_2)$. In other words
			\begin{align*}
				Z(\Psi(\ml_1,\ml_2))=\soc(Z(\ml_1)\times Z(\ml_2)).
			\end{align*}
		\end{thm}
		Thus the map $\Psi$ gives a non-recursive way of computing socles of representations induced from two ladders satisfying the permissibility condition. It would be very interesting to find a description of $\soc(\ml_1,\ml_2)$ for all pairs $(\ml_1,\ml_2)$ that is non-recursive and works by performing a single set of union-intersections.
		
		\begin{ex}
			We illustrate how Theorem \ref{thm:soc} extends similar descriptions of \cite{GL21} and \cite{Gurevich2021}. Let $\ml_1=[4,4]+[1,1]$, $\ml_2=[3,5]+[2,3]$. Then one quickly sees that $(\ml_1,\ml_2)$ that neither $(\ml_1,\ml_2)$ nor $(\ml_2^\vee,\ml_1^\vee)$ are dominant. Thus $\soc(\ml_1,\ml_2)$ cannot be computed using the results of \cite{GL21}. Since the pair is also not regular in the sense of \cite{Gurevich2021}, the description of $\soc(\ml_1,\ml_2)$ of that work can also not be used. However, $(\ml_1,\ml_2)$ is permissible in the sense of Definition \ref{def:perm} and so Theorem \ref{thm:soc} applies. We have $\Gamma(\ml_1,\ml_2)=[4,4]$ and $M(\ml_1,\ml_2)=\{([1,1],[2,3])$ which shows
			\begin{align*}
				\soc(\ml_1,\ml_2)=\Psi(\ml_1,\ml_2)=[4,4]+[3,5]+[1,3].
			\end{align*}
			This example also illustrates that juxtaposed pairs $(\Delta,\Delta^\prime)$ are not a problem for this algorithm. Computing $\soc(\ml_1,\ml_2)$ recursively using Theorem \ref{thm:lmrec} would require $4$ steps (including computing one preimage under the map $\mathcal{MW}$) instead of the single union-intersection required to compute $\Psi(\ml_1,\ml_2)$.
		\end{ex}
		
		We will prove the theorem by showing that the map $\Psi$ satisfies the recursion of Theorem \ref{thm:lmrec}. We do this in a series of lemmata, starting with the simpler case $a)$ of the recursion. Throughout this section $(\ml_1,\ml_2)$ denotes a permissible pair of ladders that is always indexed as $\ml_1=\Delta_1+\dots+\Delta_k$, $\ml_2=\Delta_1^\prime+\dots+\Delta_{k^\prime}^\prime$ in aligned form. Further, we set $\Gamma=\Gamma(\ml_1,\ml_2)$.
		\begin{lem}\label{lem:rec1}
			If $\min\ml_2\leq\min\ml_1$ then, with $\Delta^\prime=\Delta_{k^\prime}^\prime$ satisfying $\b(\Delta^\prime)=\min\ml_2$ and $\mathfrak{n}\coloneqq (\ml_1)_{\leq\Delta^\prime}$, we have
			\begin{enumerate}[a)]
				\item the pair $(\ml_1-\mathfrak{n},\ml_2-\Delta^\prime)$ is permissible, and
				\item $\Psi(\ml_1,\ml_2)=\Psi(\ml_1-\mathfrak{n},\ml_2-\Delta^\prime)+\Delta^\prime+\mathfrak{n}$.
			\end{enumerate}
		\end{lem}
		\begin{proof}
			We compute the set of matched segments $M(\ml_1-\mathfrak{n},\ml_2-\Delta^\prime)$ from which both claims will follow quickly. 
			
			Note that either $\mathfrak{n}=0$ or $\mathfrak{n}=\Delta$, where $\Delta=\Delta_k$ is the segment of $\ml_1$ with $\b(\Delta)=\min\ml_1$. If $\mathfrak{n}=0$ then $\Delta^\prime$ does not appear in $M(\ml_1,\ml_2)$, since there is no segment $\Delta_i$ of $\ml_1$ with $\Delta_i\leq\Delta^\prime$ and we have $M(\ml_1,\ml_2-\Delta^\prime)=M(\ml_1,\ml_2)$. If $\mathfrak{n}=\Delta$, then $(\Delta,\Delta^\prime)\in M(\ml_1,\ml_2)$ and we have $M(\ml_1-\Delta,\ml_2-\Delta^\prime)=M(\ml_1,\ml_2)\setminus\{(\Delta,\Delta^\prime)\}$, because $\Delta$ is the only segment of $\ml_1$ that is northwest of $\Delta^\prime$. This also shows that the cut-off $\Gamma$ does not change when removing $(\Delta,\Delta^\prime)$. Thus in either case we have $M(\ml_1-\mathfrak{n},\ml_2-\Delta^\prime)=M(\ml_1,\ml_2)\setminus\{(\Delta,\Delta^\prime)\}$.
			\begin{enumerate}[a)]
				\item By the above, we have in particular $M(\ml_1-\mathfrak{n},\ml_2-\Delta^\prime)\subseteq M(\ml_1,\ml_2)$. Lemma \ref{lem:perm} then immediately shows that $(\ml_1-\mathfrak{n},\ml_2-\Delta^\prime)$ is permissible.
				\item The identity $M(\ml_1-\mathfrak{n},\ml_2-\Delta^\prime)=M(\ml_1,\ml_2)\setminus\{(\Delta,\Delta^\prime)\}$ together with the fact that, if $\mathfrak{n}=\Delta$ we have $\Delta\subseteq\Delta^\prime$ and so $\Delta\cup\Delta^\prime+\Delta\cap\Delta^\prime=\Delta+\Delta^\prime$, implies the claim on $\Psi$ immediately.
			\end{enumerate}
		\end{proof}
		To show that $\Psi$ also satisfies the second case of the recursion, we collect some small technical results on the interplay between $\Psi$ and $\mathcal{MW}$. Since some segments of $\ml_1$, namely the singletons in the initial sequence of $\ml_1$, may become empty when going to $\ml_1^\dag$ but the matching rules for $\Psi(\ml_1,\ml_2)$ still have to consider these segments, the problem becomes more delicate if segments disappear in $\ml_1^\dag$.
		\begin{lem}\label{lem:rec2}
			Assume that $\min\ml_1<\min\ml_2$ and let $\Delta_k,\dots,\Delta_{k-l}$ be the initial sequence of $\ml_1$. 
			\begin{enumerate}[a)]
				\item If $\Delta_{k-l}<\Gamma$, which means that the initial sequence of $\ml_1$ is only in the northwest part, we have $\Gamma(\ml_1^\dag,\ml_2)=\Gamma$.
				
				If the initial sequence of $\ml_1$ extends into the southeast part and if $\Gamma$ is not a singleton, let $\Delta_\gamma$ be the nw-most segment that is southeast of $\Gamma$ in the initial sequence of $\ml_1$ such that there is no segment $\Delta^\prime$ of $\ml_2$ with $\b(\Delta^\prime)=\b(\Delta_\gamma)+1$ and $\e(\Delta^\prime)\geq\e(\Delta_\gamma)$. Then $\Gamma(\ml_1^\dag,\ml_2)={}^-\Delta_\gamma$.
				
				\item For each $i=k,\dots,k-l$ set
					\begin{align*}
						\Delta_i^*\coloneqq\begin{cases} \Delta_i\cup\Delta_j^\prime&\text{if }(\Delta_i,\Delta_j^\prime)\in M(\ml_1,\ml_2)\text{ and } b(\Delta_i)<b(\Delta_j^\prime)\\ \Delta_i &\text{else}\end{cases}.
					\end{align*}
					Then $\Delta_k^*,\dots,\Delta_{k-l}^*$ is the initial sequence of $\Psi(\ml_1,\ml_2)$.
			\end{enumerate}
		\end{lem}
		\begin{proof}
			\begin{enumerate}[a)]
				\item If $(\ml_1,\ml_2)$ is dominant, then Lemma \ref{lem:dagpair} shows that $(\ml_1^\dag,\ml_2)$ is dominant as well and so there is nothing to show. Now assume $(\ml_1,\ml_2)$ is not dominant. Then Lemma \ref{lem:dagpair} applied to $((\ml_1)_{\ngeq\Gamma},(\ml_2)_{\ngeq\Gamma})$ still shows that $\Gamma(\ml_1^\dag,\ml_2)\geq \Gamma^\dag$. Here we use that for any ladder $\ml$ and any segment $\Delta$ of $\ml$, one has $(\ml_{\ngeq\Delta})^\dag=(\ml^\dag)_{\ngeq\Delta^\dag}$. So to compute $\Gamma(\ml_1^\dag,\ml_2)$ we only need to understand what happens in the region that is southeast of $\Gamma$ when replacing $\ml_1$ by $\ml_1^\dag$. If $\Gamma$ is not in the initial sequence of $\ml_1$ then it is clear that $\Gamma(\ml_1^\dag,\ml_2)=\Gamma$ since nothing changes in the southeast region and $\Gamma^\dag=\Gamma$. If $\Gamma\leq\Delta_{k-l}$, so $\Gamma$ appears in the initial sequence of $\ml_1$, then replacing $\Delta$ by ${}^-\Delta$ for segments $\Delta$ in the southeast region can make a larger nw part of $(\ml_1,\ml_2)$ dominant (see Figure \ref{fig:dagdom} below for an illustration of the following argument).  From now on assume that $\Gamma$ is in the initial sequence of $\ml_1$ and is not a singleton. In particular, by Lemma \ref{lem:singleton}, any segment in the initial sequence of $\ml_1$ that is SE of $\Gamma$ is also not a singleton.
				
				We first show that $\Delta_\gamma$ as defined in the statement exists. Indeed assume it did not exist, then for all $\Delta\geq\Gamma$ in the initial sequence of $\ml_1$, there would be $\Delta^\prime$, a segment of $\ml_2$, such that $\Delta\leq\Delta^\prime$ and $\b(\Delta^\prime)=\b(\Delta)+1$. If $\Delta_{k-l}\neq\Delta_1$, then by the definition of an initial sequence, $\b(\Delta_{k-l-1})>\b(\Delta_{k-l})+1$ and so for any $\Delta^\prime$ as above we have $\Delta^\prime\ngeq\Delta_{k-l-1}$. Since we already know that $((\ml_1)_{\ngeq\Gamma},(\ml_2)_{\ngeq\Gamma})$ is dominant, this shows that $((\ml_1)_{\ngeq\Delta_{k-l-1}},(\ml_2)_{\ngeq\Delta_{k-l-1}})$ is dominant, which contradicts the definition of $\Gamma$ since $\Delta_{k-l-1}>\Gamma$. So this can only happen if $\Delta_{k-l}=\Delta_1$, i.e. the initial sequence of $\ml_1$ is the whole ladder. But then the existence of segments $\Delta^\prime$ as above shows that in fact $(\ml_1,\ml_2)$ is dominant, which contradicts our assumptions on $(\ml_1,\ml_2)$. In total, this shows that $\Delta_\gamma$ exists.
				
				Now we show ${}^-\Delta_\gamma=\Gamma(\ml_1^\dag,\ml_2)$. From the definition of $\Delta_\gamma$, and the previous discussion, it is clear that northwest of ${}^-\Delta_\gamma$, the pair $(\ml_1^\dag,\ml_2)$ is dominant. We need to argue that ${}^-\Delta_\gamma$ is the se-most segment of $\ml_1^\dag$ with this property, so that ${}^-\Delta_\gamma=\Gamma(\ml_1^\dag,\ml_2)$.
				
				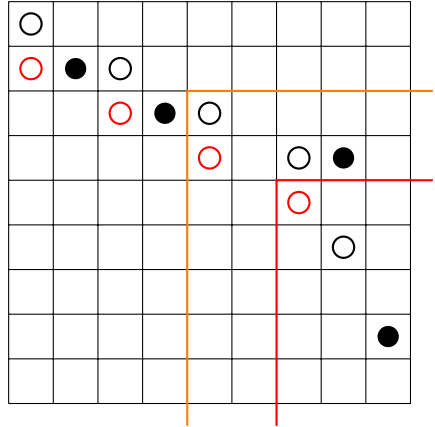
\begin{figure}[ht]
					\centering
					\begin{tikzpicture}[x=1.4em,y=1.4em, thick,color = black]
						\draw[step=1,black,thin] (0,0) grid (9,9);
						\draw [black] (0.5,8.5) circle (4pt);
						\draw [black] (2.5,7.5) circle (4pt);
    						\draw [black] (4.5,6.5) circle (4pt);
    						\draw [orange] (4,-0.5) -- (4,7) -- (9.5,7);
    						\draw [black] (6.5,5.5) circle (4pt);
    						\draw [black] (7.5,3.5) circle (4pt);
    						
    						\fill [black] (1.5,7.5) circle (4pt);
    						\fill [black] (3.5,6.5) circle (4pt);
    						\fill [black] (7.5,5.5) circle (4pt);
    						\fill [black] (8.5,1.5) circle (4pt);
    						
    						\draw [red] (0.5,7.5) circle (4pt);
						\draw [red] (2.5,6.5) circle (4pt);
    						\draw [red] (4.5,5.5) circle (4pt);
    						\draw [red] (6.5,4.5) circle (4pt);
    						\draw [red] (6,-0.5) -- (6,5) -- (9.5,5);
					\end{tikzpicture}
					\caption{An illustration of ${}^-\Delta_\gamma=\Gamma(\ml_1^\dag,\ml_2)$, with $\ml_1^\dag$ in red and the cut-off ${}^-\Delta_\gamma$ indicated by a red line.}\label{fig:dagdom}
				\end{figure}

				We do this by contradiction, so assume that there is a segment $\Delta$ SE of $\Delta_\gamma$ in $\ml_1$ such that $((\ml_1^\dag)_{\ngeq\Delta^\dag},(\ml_2)_{\ngeq\Delta^\dag})$ is dominant. Note that $(\ml_1^\dag,\ml_2)$ cannot be dominant itself, since an application of Lemma \ref{lem:dagpairrev} to the southeast region would imply that already $(\ml_1,\ml_2)$ was dominant. If $\Delta$ is not in the initial sequence of $\ml_1$ we have $\Delta^\dag=\Delta$ and $(\ml_1^\dag)_{\ngeq\Delta^\dag}=((\ml_1)_{\ngeq\Delta})^\dag$ since no segment of $\ml_1$ that is $\geq\Delta$ can become $\ngeq\Delta$ when applying $^\dag$ (here it is important that $\Delta$ is not in the initial sequence of $\ml_1$). But then $(((\ml_1)_{\ngeq\Delta})^\dag,(\ml_2)_{\ngeq\Delta})$ is dominant. Applying Lemma \ref{lem:dagpairrev} to this pair in the region that is $\geq\Gamma$, here the assumption on $\Gamma$ not being a singleton in the initial sequence of $\ml_1$ enters, then shows that already $(\ml_1,\ml_2)$ is dominant in the region that is se of $\Gamma$ and $\ngeq\Delta$. Since $((\ml_1)_{\ngeq\Gamma},(\ml_2)_{\ngeq\Gamma})$ is dominant by definition of $\Gamma$, this shows that $((\ml_1)_{\ngeq\Delta},(\ml_2)_{\ngeq\Delta})$ also has to be dominant. But this is a contradiction to the definition of $\Gamma$ since $\Delta>\Delta_\gamma\geq\Gamma$ by assumption.
				
				Thus we are left to consider the case where $\Delta$ is in the initial sequence of $\ml_1$. Here we will derive a contradiction to the definition of $\Delta_\gamma$. Now $\Delta^\dag={}^-\Delta$ and we again have $(\ml_1^\dag)_{\ngeq\Delta^\dag}=((\ml_1)_{\ngeq\Delta})^\dag$, since this holds for any ladder and any segment of it. We would run into the same contradiction to the definition of $\Gamma$ if we also had $(\ml_2)_{\ngeq\Delta^\dag}=(\ml_2)_{\ngeq\Delta}$. The only way in which this can fail is if there is a segment $\Delta^\prime$ of $\ml_2$ that is $\geq\Delta$ but $\ngeq\Delta^\dag$. This happens precisely if $\b(\Delta^\prime)=\b(\Delta)$ and $\e(\Delta^\prime)\geq\e(\Delta)$. Since $\Delta$ is in the initial sequence of $\ml_1$, the segment of $\ml_1$ immediately NW of it, of the form $\Delta_j$ for some $j=k,k-1,\dots,k-l$, satisfies $\b(\Delta_j)=\b(\Delta)-1$ and so $\b(\Delta_j^\dag)=\b(\Delta)=\b(\Delta^\prime)$ and the pair $(\Delta_j^\dag,\Delta^\prime)$ witnesses the dominance of $((\ml_1^\dag)_{\ngeq\Delta^\dag},(\ml_2)_{\ngeq\Delta^\dag})$. But then this pair of ladders can only be dominant if the same is true in the region $\ngeq\Delta_j^\dag$. Thus by induction going NW along the segments of $\ml_1$, we deduce the existence of a segment $\Delta^\prime_{j_\gamma}$ such that $\b(\Delta_\gamma)+1=\b(\Delta_\gamma^\dag)=\b(\Delta^\prime_{j_\gamma})$ and $\e(\Delta_\gamma^\dag)\leq\e(\Delta^\prime_{j_\gamma})$. This directly contradicts the definition of $\Delta_\gamma$.

				\item By construction we have $\b(\Delta_k^*)=\min\ml_1=\min\Psi(\ml_1,\ml_2)$ and $\b(\Delta_i^*)=\b(\Delta_i)$ for $i=k,\dots,k-l$ and so $\b(\Delta_{i-1}^*)=\b(\Delta_i^*)+1$ for $i=k,\dots,k-l+1$. The proof of Lemma \ref{lem:dagpair} shows that $\b(\Delta_i)<\b(\Delta_{j_i}^\prime)$ for $\Delta_i<\Gamma$ and $j_i$ such that $(\Delta_i,\Delta_{j_i}^\prime)\in M(\ml_1,\ml_2)$. By definition, this means $\Delta_i^*=\Delta_i\cup\Delta_{j_i}^\prime$ for all $i=k,\dots,k-l$ with $\Delta_i<\Gamma$. By construction of $M(\ml_1,\ml_2)$, we then have $\e(\Delta_i^*)=\e(\Delta_{j_i}^\prime)<\e(\Delta_{j_{i-1}}^\prime)=\e(\Delta_{i-1}^*)$ in the nw-region, the minimality of $\e(\Delta_i^*)$ follows immediately from the matching rules defining $\Psi(\ml_1,\ml_2)$. If $\Delta_{k-l}\leq\Gamma$ we are done once we show that this initial sequence of $\Psi(\ml_1,\ml_2)$ cannot be extended. Assume for contradiction that there is a segment $\Delta^*$ of $\Psi(\ml_1,\ml_2)$ extending the initial sequence past $\Delta_{k-l}^*$. Since the initial sequence of $\ml_1$ terminates at $\Delta_{k-l}$, $\Delta^*$ cannot be an unmatched segment of $\ml_1$ or of the form $\Delta\cup\Delta^\prime$ for $(\Delta,\Delta^\prime)\in M(\ml_1,\ml_2)$. So $\Delta^*$ is of the form $\Delta^\prime$ in $\ml_2$ or $\Delta\cap\Delta^\prime$ from a matched pair with $\b(\Delta^\prime)=\b(\Delta_{k-l})+1$ either way. But then $\e(\Delta^*)\leq\e(\Delta^\prime)$. If $\Delta_{k-l}<\Gamma$, it is matched to some segment $\Delta^\prime_j$ and $\b(\Delta_j^\prime)\geq\b(\Delta_{k-l})+1$. But then also $\e(\Delta^*)\leq\e(\Delta_j^\prime)=\e(\Delta_{k-l}^*)$ and so $\Delta^*$ cannot extend the initial sequence. If $\Delta_{k-l}=\Gamma$, $\Delta^*$ can only be a segment $\Delta^\prime$ of $\ml_2$, but by definition of $\Gamma$, such a segment cannot satisfy both $\b(\Delta^\prime)=\b(\Gamma)+1$ and $\Gamma\leq\Delta^\prime$ unless there is a segment $\Delta$ of $\ml_1$ with $\b(\Delta)=\b(\Delta^\prime)$ SE of $\Gamma$ (otherwise $\Gamma$ would not be the SE-most cut-off). But such a $\Delta$ would extend the initial sequence of $\ml_1$ past $\Gamma=\Delta_{k-l}$. In total this shows that for $\Delta_{k-l}\leq\Gamma$, the initial sequence of $\Psi(\ml_1,\ml_2)$ cannot be extended past $\Delta_{k-l}^*$ and we are done in this case.
				
				 Now assume that $\Delta_{k-l}>\Gamma$. The only case in which $\e(\Delta_{i-1}^*)>\e(\Delta_i^*)$ is not immediately obvious is if $\Delta_i^*=\Delta_i\cup\Delta_j^\prime$ for $(\Delta_i,\Delta_j^\prime)\in M(\ml_1,\ml_2)$ but $\Delta_{i-1}^*=\Delta_{i-1}$. If, for contradiction, we assumed $\e(\Delta_i^*)=\e(\Delta_j^\prime)\geq \e(\Delta_{i-1})$, then $\Delta_j^\prime$ would have been matched with $\Delta_{i-1}$ in the construction of $M(\ml_1,\ml_2)$. Indeed, since $\Delta_i,\Delta_{i-1}$ are in the initial sequence of $\ml_1$ and $\b(\Delta_j^\prime)>\b(\Delta_i)$ by construction, we would have $\b(\Delta_j^\prime)\geq\b(\Delta_i)+1=\b(\Delta_{i-1})$ in addition to $\e(\Delta_j^\prime)\geq \e(\Delta_{i-1})$. This is a contradiction, so $\e(\Delta_{i-1}^*)>\e(\Delta_i^*)$ holds also in this case. This shows that $\e(\Delta_{k-i}^*)$ is strictly increasing also in the se-region. A very similar argument using the matching rules for the construction of $\Psi(\ml_1,\ml_2)$ shows the minimality of $\e(\Delta_{k-i}^*)$ for $i=1,\dots,l$.
				 
				  To show that $\Delta_k^*,\dots,\Delta_{k-l}^*$ is the initial sequence of $\Psi(\ml_1,\ml_2)$ we are thus only left to show that it cannot be extended, i.e. that there is no segment $\Delta^*$ in $\Psi(\ml_1,\ml_2)$ such that $\b(\Delta^*)=\b(\Delta_{k-l}^*)+1$ and $\e(\Delta^*)>\e(\Delta_{k-l}^*)$. $\Delta^*$ has to be se of $\Gamma$ since otherwise $\e(\Delta^*)\leq\e(\Gamma)<\e(\Delta_{k-l}^*)$ and so it cannot extend the initial sequence. If $\Delta^*$ were a segment $\Delta$ of $\ml_1$ or $\Delta\cup\Delta^\prime$ with $(\Delta,\Delta^\prime)\in M(\ml_1,\ml_2)$, then $\Delta$ would extend the initial sequence of $\ml_1$. If it were of the form $\Delta\cap\Delta^\prime$ with $(\Delta,\Delta^\prime)\in M(\ml_1,\ml_2)$, then we would have $\e(\Delta^*)\leq\e(\Delta_{k-l}^*)$. Since all segments of $\ml_2$ in the se-region are matched in $M(\ml_1,\ml_2)$ this covers all cases and so no such $\Delta^*$ can exist.
			\end{enumerate}
		\end{proof}
		
		Lemma \ref{lem:rec2} gives us all the understanding of the relationship between $\Psi$ and $\mathcal{MW}$ we need to show that $\Psi$ also obeys the second case of the recursion in Theorem \ref{thm:lmrec} if $\Gamma$ is not a singleton in the initial sequence of $\ml_1$. 
		\begin{prop}\label{prop:main}
			Assume that $\min\ml_1<\min\ml_2$ and $\Gamma$ is not a singleton in the initial sequence of $\ml_1$. Then $(\ml_1^\dag,\ml_2)$ is permissible and
			\begin{align*}
				\Psi(\ml_1,\ml_2)^\dag&=\Psi(\ml_1^\dag,\ml_2),\text{ and} \\
				\Delta^\circ(\Psi(\ml_1,\ml_2))&=\Delta^\circ(\ml_1).
			\end{align*}
			In other words,
			\begin{align*}
				\mathcal{MW}(\Psi(\ml_1,\ml_2))=(\Psi(\ml_1^\dag,\ml_2),\Delta^\circ(\ml_1)).
			\end{align*}
		\end{prop}
		\begin{proof}
			Part b) of the previous lemma immediately shows $\Delta^\circ(\Psi(\ml_1,\ml_2))=\Delta^\circ(\ml_1)$. 
				
			We use part a) of the previous lemma to explicitly describe $M(\ml_1^\dag,\ml_2)$ from which both the permissibility of $(\ml_1^\dag,\ml_2)$ and the claim about $\Psi(\ml_1^\dag,\ml_2)$ will follow.
		
			We recall the notational conventions: we have aligned forms $\ml_1=\Delta_1+\dots+\Delta_k$, $\ml_2=\Delta_1^\prime+\dots+\Delta_{k^\prime}^\prime$ and $\Gamma(\ml_1,\ml_2)=\Gamma$. We retain also the conventions of the previous lemma: $\Delta_k,\dots,\Delta_{k-l}$ is the initial sequence of $\ml_1$, $\Delta_k^*,\dots,\Delta_{k-l}^*$ is the initial sequence of $\Psi(\ml_1,\ml_2)$.

			If $\Delta<\Gamma$ in $\ml_1$ we have to take care of the possible singletons in the initial sequence of $\ml_1$. Assume first that $\ml_1$ has initial singletons and denote them by $\Delta_k,\Delta_{k-1},\dots,\Delta_s$. By assumption, $\Delta_i<\Gamma$ for $i=k,k-1,\dots,s$. Since all these segments are in the northwest region, there are segments $\Delta_{j_i}^\prime$ of $\ml_2$ such that $(\Delta_i,\Delta_{j_i}^\prime)\in M(\ml_1,\ml_2)$ for all $i=k,\dots,s$. Since $(\ml_1,\ml_2)$ is permissible, Lemma \ref{lem:perm} ensures that for any such $i$, $\Delta_i\cup\Delta_{j_i}^\prime$ is a segment, i.e. $\b(\Delta_{j_i}^\prime)\leq \e(\Delta_i)+1=\b(\Delta_i)+1$. The proof of Lemma \ref{lem:dagpair} shows that $\b(\Delta_i)<\b(\Delta_{j_i}^\prime)$ for all $i=k,\dots,s$ and so we have $\b(\Delta_{j_i}^\prime)=\b(\Delta_i)+1$ for all such $i$. So these pairs contribute $\Delta_i\cup\Delta_{j_i}^\prime+\underbrace{\Delta_i\cap\Delta_{j_i}^\prime}_{=0}={}^+\Delta_{j_i}^\prime$ to $\Psi(\ml_1,\ml_2)$. By Lemma \ref{lem:rec2} b), all these ${}^+\Delta_{j_i}^\prime$ are in the initial sequence of $\Psi(\ml_1,\ml_2)$ and so contribute ${}^-({}^+\Delta_{j_i}^\prime)=\Delta_{j_i}^\prime$ to $\Psi(\ml_1,\ml_2)^\dag$ and these are all segments of $\Psi(\ml_1,\ml_2)^\dag$ with $\b\leq\b(\Delta_s)+1$. Conversely, in $\ml_1^\dag$ all segments $\Delta_k,\Delta_{k-1},\dots,\Delta_s$ disappear as $\Delta_i^\dag=0$ for $i=k,\dots,s$. This leaves no segments in $\ml_1^\dag$ that may be matched to any of $\Delta_{j_k}^\prime,\dots,\Delta_{j_s}^\prime$. Accordingly, these $\Delta_{j_i}^\prime$ appear in $\Psi(\ml_1^\dag,\ml_2)$ as they are, just as they do in $\Psi(\ml_1,\ml_2)^\dag$. All segments $\Delta$ of $\ml_1$ with $\Delta_s<\Delta<\Gamma$ satisfy $\Delta^\dag\neq 0$ and are matched in $M(\ml_1,\ml_2)$ to segments $\Delta^\prime$ of $\ml_2$ with $\Delta^\prime>\Delta_{j_s}^\prime$. Again by an argument as in the proof of Lemma \ref{lem:dagpair} we know that also $\Delta^\dag\leq\Delta^\prime$ for all $(\Delta,\Delta^\prime)\in M(\ml_1,\ml_2)$ with $\Delta_s<\Delta<\Gamma$. From this one quickly sees that in fact $(\Delta,\Delta^\prime)\in M(\ml_1,\ml_2)$ if and only if $(\Delta^\dag,\Delta^\prime)\in M(\ml_1^\dag,\ml_2)$ for $\Delta_s<\Delta<\Gamma$, i.e. the matchings do not change in this region. If $\ml_1$ contains no initial singletons, then the same argument shows that the matching in the northwest region does not change when replacing $\ml_1$ by $\ml_1^\dag$.
			
			If $\Gamma$ does not appear in the initial sequence of $\ml_1$, then Lemma \ref{lem:rec2} a) shows that $\Gamma(\ml_1^\dag,\ml_2)=\Gamma$ and clearly $(\ml_1^\dag,\ml_2)$ is identical to $(\ml_1,\ml_2)$ in the region southeast of $\Gamma$ in this case. Thus, if $\Gamma$ does not appear in the initial sequence of $\ml_1$, this describes $M(\ml_1^\dag,\ml_2)$ completely.
			
			Now assume that $\Gamma$ appears in the initial sequence and by the assumption of the proposition is not a singleton. Then $\Gamma(\ml_1^\dag,\ml_2)={}^-\Delta_\gamma$ is as described in Lemma \ref{lem:rec2} a). 
			
			If $\Gamma\leq\Delta<\Delta_\gamma$ for $\Delta$ in $\ml_1$, we use the description of ${}^-\Delta_\gamma=\Gamma(\ml_1^\dag,\ml_2)$ of Lemma \ref{lem:rec2} a). See also Figure \ref{fig:dagdom} above for a picture of the situation. In this region the matching will shift by $1$ when going from $\ml_1$ to $\ml_1^\dag$. For notational simplicity we denote by $S$ the map that sends a segment $\Delta$ of a ladder $\ml$ to the segment immediately SouthEast of $\Delta$ in $\ml$. With this, for any segment $\Delta$ of $\ml_1$ with $\Gamma\leq\Delta<\Delta_\gamma$ we have $(S(\Delta),\Delta^\prime)\in M(\ml_1,\ml_2)$ if and only if $({}^-\Delta,\Delta^\prime)\in M(\ml_1^\dag,\ml_2)$, which is immediately clear from the proof of a) of the previous lemma and the explicit description of ${}^-\Delta_\gamma=\Gamma(\ml_1^\dag,\ml_2)$ in Lemma \ref{lem:rec2}. Finally note that ${}^-\Delta_\gamma$ itself does not appear in $M(\ml_1^\dag,\ml_2)$ since it is equal to $\Gamma(\ml_1^\dag,\ml_2)$. 
			
			If $\Delta>\Delta_\gamma$ in $\ml_1$ and $(\Delta,\Delta^\prime)\in M(\ml_1,\ml_2)$ for some $\Delta^\prime$ in $\ml_2$ then there are two options. If $\Delta$ is not in the initial sequence of $\ml_1$, i.e. is SE of $\Delta_{k-l}$, then $\Delta^\dag=\Delta$ and $(\Delta,\Delta^\prime)\in M(\ml_1^\dag,\ml_2)$. Otherwise, so if $\Delta=\Delta_i$ for some $\gamma>i\geq k-l$ we have $({}^-\Delta_i,\Delta^\prime)\in M(\ml_1^\dag,\ml_2)$ if $\b(\Delta_i)<\b(\Delta^\prime)$ and otherwise we have $({}^-\Delta_{i+1},\Delta^\prime)\in M(\ml_1^\dag,\ml_2)$. In other words, if the matched segments $\Delta,\Delta^\prime$ have the same beginning, then changing $\Delta$ to ${}^-\Delta$ shifts the matching by $1$. This is clear, since if $\b(\Delta_i)<\b(\Delta^\prime)$ we still have ${}^-\Delta_i\leq\Delta^\prime$ and so $\Delta^\prime$ still matches to ${}^-\Delta_i$ in $M(\ml_1^\dag,\ml_2)$ and if $\b(\Delta_i)=\b(\Delta^\prime)$ then ${}^-\Delta_i\nleq\Delta^\prime$ but now ${}^-\Delta_{i+1}\leq\Delta^\prime$. One sees that ${}^-\Delta_{i+1}$ is available for $\Delta^\prime$ to match to in $(\ml_1^\dag,\ml_2)$ by a simple induction using the fact that $\Delta_i$ matches with $\Delta^\prime$ in $M(\ml_1,\ml_2)$. This gives a complete description of $M(\ml_1^\dag,\ml_2)$ also in this region.
			
			This complete description of $M(\ml_1^\dag,\ml_2)$ shows that for all $(\Delta^\dag,\Delta^\prime)$ in that matching, also $\Delta^\dag\cup\Delta^\prime$ is a segment and so $(\ml_1^\dag,\ml_2)$ is permissible by Lemma \ref{lem:perm}.
			
			Now if $({}^-\Delta,\Delta^\prime)\in M(\ml_1^\dag,\ml_2)$ with $\Gamma\leq\Delta<\Delta_\gamma$ then $(S(\Delta),\Delta^\prime)\in M(\ml_1,\ml_2)$ and $\b(\Delta^\prime)=\b(S(\Delta))=\b({}^-\Delta)$. Thus both pairs are unlinked so that in $\Psi(\ml_1^\dag,\ml_2)$ resp. $\Psi(\ml_1,\ml_2)$ they contribute ${}^-\Delta+\Delta^\prime$ resp. $S(\Delta)+\Delta^\prime$. The same argument also takes care of the pairs with equal beginning in the region SE of $\Delta_\gamma$. If $\Delta,\Delta^\prime$ are matched with $\b(\Delta)<\b(\Delta^\prime)$, then $({}^-\Delta)\cap\Delta^\prime=\Delta\cap\Delta^\prime$ shows that also for the other matched pairs in the SE region the contributions match, since by Lemma \ref{lem:rec2} b) the intersection will not appear in the initial sequence of $\Psi(\ml_1,\ml_2)$. Furthermore, for any $\Delta^*=\Delta\cup\Delta^\prime$ in the initial sequence of $\Psi(\ml_1,\ml_2)$ (see Lemma \ref{lem:rec2} b)) we have ${}^-(\Delta^*)=({}^-\Delta)\cup\Delta^\prime$ and so we also see that in the northwest region, taking $\Psi$ and applying $()^\dag$ are compatible in the way we want. In total this shows:
			
			\begin{equation}
				\Psi(\ml_1,\ml_2)^\dag=\Psi(\ml_1^\dag,\ml_2).
			\end{equation}
				
		\end{proof}
		
		To finish the proof of Theorem \ref{thm:soc} we still need to understand the case where $\min\ml_1<\min\ml_2$ and $\Gamma(\ml_1,\ml_2)$ is a singleton in the initial sequence of $\ml_1$. First we show that in this case, $\Gamma(\ml_1,\ml_2)$ has to be the northwestmost segment of $\ml_1$, so that the northwest region of $(\ml_1,\ml_2)$ is empty.
		\begin{lem}\label{lem:gammasing}
			Assume that $\min\ml_1<\min\ml_2$ and $\Gamma(\ml_1,\ml_2)$ is in the initial sequence of $\ml_1$ and is a singleton. Then $\Gamma(\ml_1,\ml_2)=[\min\ml_1,\min\ml_1]$.
		\end{lem}
		\begin{proof}
			Let $\Delta_k,\dots,\Delta_{k-l}$ be the initial sequence of $\ml_1$ with $\Gamma\coloneqq\Gamma(\ml_1,\ml_2)=\Delta_{g-1}$, $k-l\leq g-1\leq k$. We need to show that $g-1=k$ so that $\Gamma$ is the segment $\Delta_k$ of $\ml_1$ with $\b(\Gamma)=\min\ml_1$. Assume for contradiction that $g-1<k$. By the definition of $\Gamma$, we know that $((\ml_1)_{\ngeq\Gamma},(\ml_2)_{\ngeq\Gamma})$ is dominant and we have $(\ml_1)_{\ngeq\Gamma}=\Delta_g+\Delta_{g+1}+\dots+\Delta_k$ in aligned form. Thus for all $i=g,\dots,k$ there is a segment $\Delta^\prime_{j_i}\ngeq\Gamma$ of $\ml_2$ such that $\Delta_i\leq\Delta_{j_i}^\prime$. We will derive a contradiction to $\min\ml_1<\min\ml_2$ by showing that $\b(\Delta_k)=\b(\Delta_{j_k}^\prime)$.
			
			By Lemma \ref{lem:singleton} we know that $\Delta_i$ is a singleton for all $i=g,\dots,k$. In particular, $\Delta_{j_g}^\prime\geq \Delta_g$ and $\Delta_{j_g}^\prime\ngeq\Gamma=[\b(\Delta_g)+1,\b(\Delta_g)+1]$ implies that $\b(\Delta_{j_g}^\prime)=\b(\Delta_g)$ or $\e(\Delta_{j_g}^\prime)=\e(\Delta_g)$. In the latter case we deduce $\b(\Delta_g)\leq\b(\Delta_{j_g}^\prime)\leq\e(\Delta_{j_g}^\prime)=\e(\Delta_g)=\b(\Delta_g)$ and so $\Delta_{j_g}^\prime=\Delta_g$. Either way, we thus have $\b(\Delta_{j_g}^\prime)=\b(\Delta_g)$. For $i>g$ the same argument shows by induction that $\b(\Delta_{j_i}^\prime)=\b(\Delta_i)$ for $i=g,\dots,k$ and so in particular $\min\ml_1=\b(\Delta_k)=\b(\Delta_{j_k}^\prime)$. This contradicts $\b(\Delta_k)=\min\ml_1<\min\ml_2$ and so $\Gamma=\Delta_k$ as claimed.
		\end{proof}
		\begin{proof}[Proof of Theorem \ref{thm:soc}]
			This section already established all the ingredients to make the proof straightforward. We prove it by induction on the total length of $\ml_1+\ml_2$, i.e. the sum of the lengths of all segments. If one of $\ml_1,\ml_2$ is $0$, then Theorem \ref{thm:soc} is trivial. So we may assume that both ladders are non-zero.
			
			First we prove the result for all dominant permissible pairs $(\ml_1,\ml_2)$. So assume in the following that $(\ml_1,\ml_2)$ is dominant and note that Lemma \ref{lem:dagpair} ensures that the recursion of Theorem \ref{thm:lmrec} stays within the class of dominant pairs. As the base case, if this total length is $0$, both $\ml_1,\ml_2=0$ and the claim is obvious since all representations involved are the trivial one.
			
			If $\min\ml_1\geq\min\ml_2$ then Lemma \ref{lem:rec1} shows, with the notation of that lemma, $\Psi(\ml_1-\mathfrak{n},\ml_2-\Delta^\prime)+\Delta^\prime+\mathfrak{n}=\Psi(\ml_1,\ml_2)$. On the other hand, by Theorem \ref{thm:lmrec} the relation holds for the socles and by induction, $Z(\Psi(\ml_1-\mathfrak{n},\ml_2-\Delta^\prime))=\soc(Z(\ml_1-\mathfrak{n})\times Z(\ml_2-\Delta^\prime))$, proving also $Z(\Psi(\ml_1,\ml_2))=\soc(Z(\ml_1)\times Z(\ml_2))$ in this case.
			
			If $\min\ml_1<\min\ml_2$, then by Proposition \ref{prop:main} $(\ml_1^\dag,\ml_2)$ is permissible and $\mathcal{MW}(\Psi(\ml_1,\ml_2))=(\Psi(\ml_1^\dag,\ml_2),\Delta^\circ(\ml_1))$. But since $\ml_1^\dag+\ml_2$ has strictly smaller total length than $\ml_1+\ml_2$ and $\Psi$ clearly preserves total length, we know by induction that $Z(\Psi(\ml_1^\dag,\ml_2))=\soc(Z(\ml_1^\dag)\times Z(\ml_2))$. Again invoking \ref{thm:lmrec} then shows that
			\begin{align*}
				\mathcal{MW}(\Psi(\ml_1,\ml_2))= (\soc(\ml_1^\dag,\ml_2),\Delta^\circ(\ml_1))=\mathcal{MW}(\soc(\ml_1,\ml_2)).
			\end{align*}
			Since the map $\mathcal{MW}$ is injective, we obtain $Z(\Psi(\ml_1,\ml_2))=\soc(Z(\ml_1)\times Z(\ml_2))$ also in this case.
			
	This proves the result for all dominant pairs. In general, we do the same induction to show that $\Psi$ agrees with $\soc$. The only case that cannot be treated in the same way, is the case where $\Gamma=\Gamma(\ml_1,\ml_2)$ is a singleton in the initial sequence and $\min\ml_1<\min\ml_2$. But then by Lemma \ref{lem:gammasing} we know that $\Gamma$ is the NW-most segment of $\ml_1$ and so by Lemma \ref{lem:dualdom} the pair $(\ml_2^\vee,\ml_1^\vee)$ is dominant and it is permissible by definition of permissibility. But then we know from the above induction that $\Psi(\ml_2^\vee,\ml_1^\vee)=\soc(\ml_2^\vee,\ml_1^\vee)$. From the definition of $\Psi$, it is clear that replacing $(\ml_1,\ml_2)$ by $(\ml_2^\vee,\ml_1^\vee)$ simply exchanges orientations east/west resp. north/south and swaps the roles of $\ml_1$ and $\ml_2$. Since union-intersection commutes with $()^\vee$, we see $\Psi(\ml_2^\vee,\ml_1^\vee)=\Psi(\ml_1,\ml_2)^\vee$. Finally, by standard properties of the interplay between taking contragredients, socles and parabolic induction, we have
			\begin{align*}
				\soc(Z(\ml_2^\vee)\times Z(\ml_1^\vee))^\vee&=\cos((Z(\ml_2^\vee)\times Z(\ml_1^\vee))^\vee)=\cos(Z(\ml_2^\vee)^\vee\times Z(\ml_1^\vee)^\vee)\\&=\cos(Z(\ml_2)\times Z(\ml_1))=\soc(Z(\ml_1)\times Z(\ml_2)),
			\end{align*}
			where we have used the \textit{cosocle} $\cos$ of a representation, the largest semi-simple quotient of a representation. It is easy to see that $\soc(\pi^\vee)=\cos(\pi)^\vee$ holds for any representation $\pi$. In the final step of the above computation, we also used the fact that
			\begin{align*}
				\cos(\pi_2\times\pi_1)\cong\soc(\pi_1\times\pi_2),
			\end{align*}
			holds for any two irreducible representations $\pi_1,\pi_2$. See \cite[Section 2.6]{LM16} for a short proof of this.  
			
			Putting everything together, we obtain $Z(\Psi(\ml_1,\ml_2))=\soc(Z(\ml_1)\times Z(\ml_2))$ and this finishes the proof of the theorem.
	
		\end{proof}
		
	\section{Relation to RSK}\label{sec:appl}
	The goal of this section is to put the constructions of the map $\Psi$ into a more classical combinatorial context. Concretely, we will use Viennot's shadow construction of the RSK-correspondence to describe the inverse of $\Psi$ (or more properly to describe pre-image $\Psi^{-1}(\m)$ for a given multisegment $m$). Analogously to the way we constructed $\Psi$, we will do this by applying the shadow construction with different orientations.
	
	The constructions of this section were sparked by a conjecture communicated to us by Erez Lapid and were the motivation and inspiration for constructing the map $\Psi$. We describe first the setting that we will work in for the remainder of this article and then state the main result of this section.

	\begin{defi}
		A pair $(\ml_1,\ml_2)$ of ladders is called \textit{regular}, if there are no multiplicities among the $\b(\Delta)$'s and no multiplicities among the $\e(\Delta)$'s for $\Delta$ in $\ml_1+\ml_2$ and $\b(\Delta)\leq\e(\Delta^\prime)$ for all $\Delta,\Delta^\prime$ in $\ml_1+\ml_2$.
	\end{defi}
	Regular pairs are often simpler to work with from a combinatorial perspective, because in this case one can always perform union-intersection if $\Delta\leq\Delta^\prime$ for two segments. In \cite{Gurevich2021} the problem of characterising $\soc(Z(\ml_1)\times Z(\ml_2))$ (and indeed all other irreducible subquotients) was investigated from the point of view of quantum groups and in \cite[Theorem 6.11]{Gurevich2021} a similar construction to ours was given in somewhat different terms for regular pairs.
		\begin{rmk}
			The question that sparked the work on this article was Theorem \ref{thm:n+1} which was communicated to us as a conjecture by Erez Lapid and which only refers to regular pairs $(\ml_1,\ml_2)$. After (conjectorially) finding the combinatorial construction $\Psi$ for $\soc(\ml_1,\ml_2)$ in the regular setting a technical problem is that the Lapid--Mínguez recursion of Theorem \ref{thm:lmrec} doesn't preserve regularity and so one has to extend this description to a class of pairs that is stable under this recursion. As Lemma \ref{lem:rec1} and Proposition \ref{prop:main} show, permissibility is precisely the right notion for this.
		\end{rmk}
	 
		When considering regular pairs it is no serious loss of generality to restrict to the setting where $\{\b(\Delta)\colon\Delta\text{ in }\ml_1+\ml_2\}=\{1,\dots,n\}=[n]$ and $\{\e(\Delta)\colon\Delta\text{ in }\ml_1+\ml_2\}=\{n,\dots,2n-1\}$. The pair $(\ml_1,\ml_2)$ is then determined by the sets $I=\{\b(\Delta)\colon\Delta\text{ in }\ml_1\}$ and $J=\{\e(\Delta)-n+1\colon\Delta\text{ in }\ml_1\}$, which are subsets of equal size of $[n]$. More precisely, given $I,J\subseteq[n]$ of equal size, we can define a ladder $\ml_{I,J}$ by setting
		\begin{align*}
			\ml_{I,J}=\sum_{i=1}^k[a_i,b_i+n-1],
		\end{align*}
		where $I=\{a_1>\dots>a_k\}$ and $J=\{b_1>\dots>b_k\}$. Denoting by $(\cdot)^c$ the complement in $[n]$, this gives a regular pair $(\ml_{I,J},\ml_{I^c,J^c})$ of ladders for any pair $(I,J)$ of subsets of $[n]$ of equal size and it is easy to see that any regular pair of ladders with $\b(\ml_1+\ml_2)=[n]$ and $\e(\ml_1+\ml_2)=\{n,\dots,2n-1\}$ is of this form.
		
		For convenience we set $X_n\coloneqq\{(I,J)\colon I,J\subseteq[n],\,|I|=|J|\}$ and identify elements of this set with pairs of ladders as described above.
		
		In similar fashion, if $\sigma\in S_n$ is a permutation we obtain a multisegment $\m_\sigma$ by setting
		\begin{align*}
			\m_\sigma=\sum_{i=1}^n[i,\sigma(i)+n-1].
		\end{align*}
		We think of $\m_\sigma$ essentially as the permutation matrix of $\sigma$, the shift by $n-1$ in the second coordinate is simply to ensure that these are really segments.
		\begin{defi}
			We define a map $\psi:X_n\rightarrow S_n$ as follows. Given $(I,J)\in X_n$, let $\psi(I,J)$ be the unique $321$-avoiding permutation $\sigma$ such that
			\begin{align*}
				\soc(Z(\ml_{I,J})\times Z(\ml_{I^c,J^c}))=Z(\m_\sigma).
			\end{align*}
		\end{defi}
		The fact that the socle in question can be described by a $321$-avoiding permutation follows from \cite[Corollary 4.13]{Gurevich2020}. An immediate consequence of Theorem \ref{thm:soc} is given next.
		\begin{cor}\label{cor:psi}
			Let $(I,J)\in X_n$, then $\m_{\psi(I,J)}=\Psi(\ml_{I,J},\ml_{I^c,J^c}).$
		\end{cor} 
		Thus the map $\psi$ is the same as $\Psi$ under the above identifications of pairs of sets $(I,J)\in X_n$ with regular pairs of ladders and permutations $\sigma\in S_n$ with certain multisegments. The goal of this section is to prove the following theorem, which counts how many regular pairs $(\ml_1,\ml_2)$ have a specified socle $Z(\m_\sigma)$. Surprisingly, the answer to this is independent of the $321$-avoiding permutation $\sigma$! We denote the set of $321$-avoiding permutations by $S_n^{321}$. It is a classical result that $|S_n^{321}|=\frac{1}{n+1}\binom{2n}{n}$, the $n$-th Catalan number. Note also $|X_n|=\binom{2n}{n}$. To see this, for any $n$-element subset $A$ of $[2n]$, the sets $I=A\cap[n]$ and $J=\{n+1,\dots,2n\}\setminus A$ have the same size and so $A\mapsto(I,J-n)\in X_n$ and this is a bijection. 
		
		\begin{thm}\label{thm:n+1}
			For $n\in\N$ the map $\psi:X_n\rightarrow S_n^{321}$ is $(n+1)$-to-$1$. That is, for each $\sigma\in S_n^{321}$ we have $|\psi^{-1}(\sigma)|=n+1$.
		\end{thm}
		We will prove this by giving a combinatorial description of $\psi^{-1}(\sigma)$ in terms of the Robinson-Schensted-Knuth correspondence. 
		
		\subsection{Variants of RSK}\label{sec:rsk}
		
		The classical Robinson--Schensted--Knuth (RSK) correspondence is a bijection between multisets of integer pairs $\m$ and pairs of semi-standard Young tableaux (SSYT) of the same shape $(P,Q)$, see \cite{Fulton1996} for standard properties of this map. In particular, this gives a map $\mathcal{R}^{SE}$ (the reason for this notation will become clear in the following) taking in multisegments $\m$ and producing pairs of SSYT of the same shape. We will use Viennot's geometric construction (also known as the matrix-ball construction) of RSK, see \cite{Fulton1996} Chapter~4 and Appendix~A for details on this in general. 
		\begin{ex}\label{ex:rskse}
		To illustrate this construction, we consider the $321$-avoiding permutation $\sigma=241579368$ in $S_9$.
		
		\begin{figure}[ht]
			\centering
			\begin{tikzpicture}[x=1.4em,y=1.4em, thick,color = black]
			\draw[step=1,black,thin] (0,0) grid (9,9);
			\draw [black] (1.5,8.5) circle (4pt);
			\draw [black] (3.5,7.5) circle (4pt);
    			\draw [black] (4.5,5.5) circle (4pt);
    			\draw [black] (6.5,4.5) circle (4pt);
    			\draw [black] (8.5,3.5) circle (4pt);
    			\draw [black] (0.5,6.5) circle (4pt);
    			\draw [black] (2.5,2.5) circle (4pt);
    			\draw [black] (5.5,1.5) circle (4pt);
    			\draw [black] (7.5,0.5) circle (4pt);
    			\node at (0.5,9.5) {$1$};
    			\node at (1.5,9.5) {$2$};
    			\node at (2.5,9.5) {$3$};
    			\node at (3.5,9.5) {$4$};
    			\node at (4.5,9.5) {$5$};
    			\node at (5.5,9.5) {$6$};
    			\node at (6.5,9.5) {$7$};
    			\node at (7.5,9.5) {$8$};
    			\node at (8.5,9.5) {$9$};
    			\node at (-0.5,8.5) {$1$};
    			\node at (-0.5,7.5) {$2$};
    			\node at (-0.5,6.5) {$3$};
    			\node at (-0.5,5.5) {$4$};
    			\node at (-0.5,4.5) {$5$};
    			\node at (-0.5,3.5) {$6$};
    			\node at (-0.5,2.5) {$7$};
    			\node at (-0.5,1.5) {$8$};
    			\node at (-0.5,0.5) {$9$};
		\end{tikzpicture}
		\caption{The permutation $\sigma=241579368$ visualised in the plane.}
		\label{fig:experm}
		\end{figure}
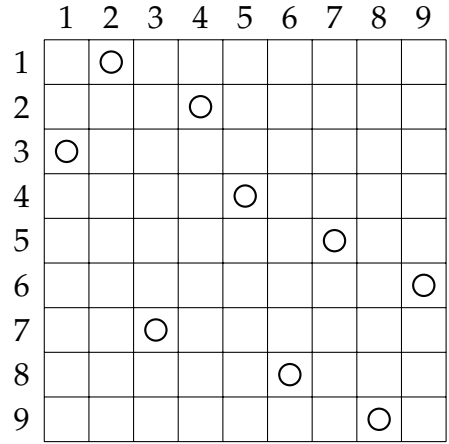
		
		The first set of shadow lines in the construction of $\mathcal{R}^{SE}(\sigma)$ then looks as illustrated by Figure \ref{fig:shad1}.
		
		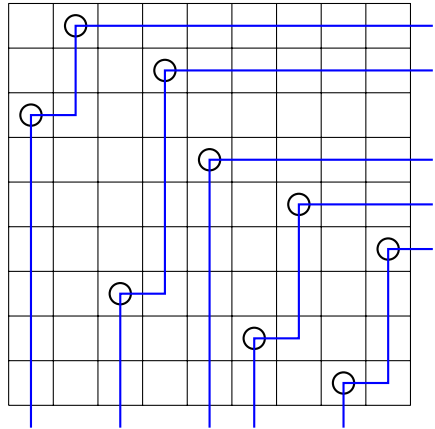
\begin{figure}[ht]
			\centering
			\begin{tikzpicture}[x=1.4em,y=1.4em, thick,color = black]
			\draw[step=1,black,thin] (0,0) grid (9,9);
			\draw [black] (1.5,8.5) circle (4pt);
			\draw [black] (3.5,7.5) circle (4pt);
    			\draw [black] (4.5,5.5) circle (4pt);
    			\draw [black] (6.5,4.5) circle (4pt);
    			\draw [black] (8.5,3.5) circle (4pt);
    			\draw [black] (0.5,6.5) circle (4pt);
    			\draw [black] (2.5,2.5) circle (4pt);
    			\draw [black] (5.5,1.5) circle (4pt);
    			\draw [black] (7.5,0.5) circle (4pt);
    			\draw [blue] (0.5,-0.5) -- (0.5,6.5) -- (1.5,6.5) -- (1.5,8.5) -- (9.5,8.5);
    			\draw [blue] (2.5,-0.5) -- (2.5,2.5) -- (3.5,2.5) -- (3.5,7.5) -- (9.5,7.5);
    			\draw [blue] (4.5,-0.5) -- (4.5,5.5) -- (9.5,5.5);
    			\draw [blue] (5.5,-0.5) -- (5.5,1.5) -- (6.5,1.5) -- (6.5,4.5) -- (9.5,4.5);
    			\draw [blue] (7.5,-0.5) -- (7.5,0.5) -- (8.5,0.5) -- (8.5,3.5) -- (9.5,3.5);
		\end{tikzpicture}
		\caption{The first step in the computation of $\mathcal{R}^{SE}(\sigma)$.}
		\label{fig:shad1}
		\end{figure}
		
		From Figure \ref{fig:shad1} we can read off the first rows of $\mathcal{R}^{SE}(\sigma)=(P,Q)$: that of $P$ is made up of the row indices of the blue lines leaving the square on the right and the first row of $Q$ is made up of the column indices of the lines entering the square from the bottom. So the first row of $P$ is $(1,2,4,5,6)$ and that of $Q$ is $(1,3,5,6,8)$.
		
		Draw a circle in each square containing an elbow of a blue line and delete all circles of the original permutation: this produces a new constellation of circles and we repeat this process to get the next rows of $P,Q$. Since $\sigma$ is $321$-avoiding, the process terminates after this step.
		
		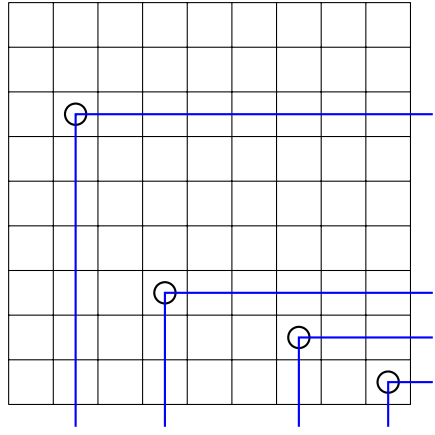
\begin{figure}[ht]
			\centering
			\begin{tikzpicture}[x=1.4em,y=1.4em, thick,color = black]
			\draw[step=1,black,thin] (0,0) grid (9,9);
			\draw [black] (1.5,6.5) circle (4pt);
			\draw [black] (3.5,2.5) circle (4pt);
    			\draw [black] (6.5,1.5) circle (4pt);
    			\draw [black] (8.5,0.5) circle (4pt);
    			\draw [blue] (1.5,-0.5) -- (1.5,6.5) -- (9.5,6.5);
    			\draw [blue] (3.5,-0.5) -- (3.5,2.5) -- (9.5,2.5);
    			\draw [blue] (6.5,-0.5) -- (6.5,1.5) -- (9.5,1.5);
    			\draw [blue] (8.5,-0.5) -- (8.5,0.5) -- (9.5,0.5);
		\end{tikzpicture}
		\caption{The second step in the computation of $\mathcal{R}^{SE}(\sigma)$.}
		\end{figure}
		
		The second rows of $P,Q$ are therefore given by $3,7,8,9,$ and $2,4,7,9,$ respectively. In total this shows:
		\begin{align*}
			\mathcal{R}^{SE}(\sigma)=\left(\begin{ytableau} 1&2&4&5&6\\3&7&8&9 \end{ytableau},\begin{ytableau} 1&3&5&6&8\\2&4&7&9\end{ytableau}\right).
		\end{align*}
		\end{ex}
		This construction leaves some choices regarding the orientation of the blue shadow lines. We will need a different choice of orientations where, going bottom to top, lines enter from the left and exit at the top, which we denote by $\mathcal{R}^{NW}$. This produces tableaux with strictly decreasing rows and weakly decreasing columns, which, following \cite{GL21}, we will refer to as \textit{inverted} semi-standard Young tableaux (ISSYT). The map $\mathcal{R}^{NW}$ is just the map $\mathcal{RSK}$ of \cite{GL21}. For clarity, we illustrate the construction on the same example as above.
		\begin{ex}\label{ex:rsknw}
			Again take $\sigma=241579368$, then the first step in the construction of $\mathcal{R}^{NW}(\sigma)=(P,Q)$ yields the shadow lines
			
		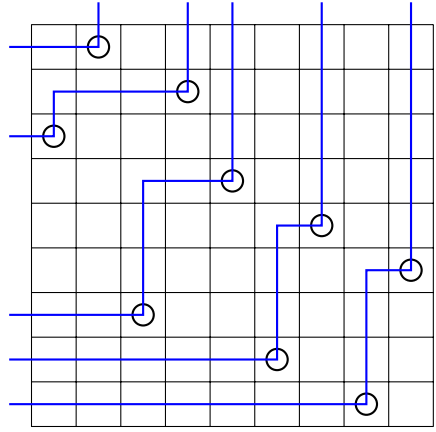
\begin{figure}[ht]
			\centering
			\begin{tikzpicture}[x=1.4em,y=1.4em, thick,color = black]
			\draw[step=1,black,thin] (0,0) grid (9,9);
			\draw [black] (1.5,8.5) circle (4pt);
			\draw [black] (3.5,7.5) circle (4pt);
    			\draw [black] (4.5,5.5) circle (4pt);
    			\draw [black] (6.5,4.5) circle (4pt);
    			\draw [black] (8.5,3.5) circle (4pt);
    			\draw [black] (0.5,6.5) circle (4pt);
    			\draw [black] (2.5,2.5) circle (4pt);
    			\draw [black] (5.5,1.5) circle (4pt);
    			\draw [black] (7.5,0.5) circle (4pt);
    			\draw [blue] (-0.5,0.5) -- (7.5,0.5) -- (7.5,3.5) --(8.5,3.5) -- (8.5,9.5);
    			\draw [blue] (-0.5,1.5) -- (5.5,1.5) -- (5.5,4.5) -- (6.5,4.5) -- (6.5,9.5);
    			\draw [blue] (-0.5,2.5) -- (2.5,2.5) -- (2.5,5.5) -- (4.5,5.5) -- (4.5,9.5);
    			\draw [blue] (-0.5,6.5) -- (0.5,6.5) -- (0.5,7.5) -- (3.5,7.5) -- (3.5,9.5);
    			\draw [blue] (-0.5,8.5) -- (1.5,8.5) -- (1.5,9.5);
		\end{tikzpicture}
		\caption{The first step in the computation of $\mathcal{R}^{NW}(\sigma)$.}
		\end{figure}
		
			Here too, we read off the first rows of $P,Q$ by the row resp. column indices of the blue lines. This gives $(9,8,7,3,1)$ for the first line of $P$ and $(9,7,5,4,2)$ for that of $Q$. Again placing circles in the positions where the blue lines have elbows we create a new constellation and repeat this process. In total, the resulting tableaux are 
			\begin{align*}
				\mathcal{R}^{NW}(\sigma)=\left(\begin{ytableau} 9&8&7&3&1\\6&5&4&2\end{ytableau},\begin{ytableau} 9&7&5&4&2\\8&6&3&1\end{ytableau}\right).
			\end{align*}
		\end{ex}
		If $\sigma\in S_n$ is $321$-avoiding, then by standard properties of the RSK-correspondence, $\mathcal{R}^{SE}(\sigma),\mathcal{R}^{NW}(\sigma)$ will consist of (inverted) standard tableaux with at most $2$ rows and entries in $\{1,\dots,n\}$. This allows us to identify each tableau with just one of its rows (since the other row is then uniquely determined), i.e. with a subset of $\{1,\dots,n\}$.
		
		Setting $\mathcal{R}^{NW}(\sigma)=(P,Q)$, we identify $(P,Q)$ with the pair $(I,J)\in X_n$ where $I$ resp. $J$ is the second row of $P$ resp. $Q$ (which might be empty). Similarly, if $\mathcal{R}^{SE}(\sigma)=(\widetilde{P},\widetilde{Q})$, we identify this with a pair $(\widetilde{I},\widetilde{J})\in X_n$ where we take the first row of $\widetilde{P}$ resp. $\widetilde{Q}$. Restricting to $321$-avoiding permutations, we will thus view $\mathcal{R}^{NW},\mathcal{R}^{SE}$ as maps $S_n^{321}\rightarrow X_n$. 
		
		Translating \cite[Theorem 4.3]{GL21} to our setting immediately shows
		\begin{prop}\label{prop:nwrsk}
			If $\sigma\in S_n^{321}$ and $\mathcal{R}^{NW}(\sigma)=(I,J)$, then $\psi(I,J)=\sigma$.
		\end{prop}
		\begin{proof}
			This is just a special case of a reformulation of \cite[Theorem 4.3]{GL21}: $\mathcal{R}^{NW}(\sigma)=(I,J)$ is precisely saying that $\mathcal{K}(\m_\sigma)=(\ml_{I^c,J^c},\ml_{I,J})$ in the notation of \cite{GL21} and therefore
			\begin{align*}
				\soc(Z(\ml_{I,J})\times Z(\ml_{I^c,J^c}))=Z(\m_\sigma),
			\end{align*}
			which by definition means $\psi(I,J)=\sigma$.
		\end{proof}
		Passing to contragredients, it is easy to see that also $\mathcal{R}^{SE}$ is a section of the map $\psi$ in this sense. See also the proof of Theorem \ref{thm:soc} for the relevant argument. This also shows why the identification of tableaux with subsets of $[n]$ has to be switched from SYT to ISYT.
		\begin{rmk}
			The previous proposition shows how the map $\Psi$ is constructed precisely to be the inverse of the RSK map $\mathcal{R}^{NW}$ in the northwest region and of $\mathcal{R}^{SE}$ in the southeast region. Indeed, the way $\Psi$ is constructed, matching points and switching their coordinates, is a well-known way of constructing the inverse to $\mathcal{R}^{NW}$.
		\end{rmk} 
		\subsection{Mixed RSK}
		To describe the whole set $\psi^{-1}(\sigma)$ for $\sigma\in S_n^{321}$ we will now 'mix' $\mathcal{R}^{NW},\mathcal{R}^{SE}$ in a similar way to the construction of $\Psi$. Note that the maps $\mathcal{R}^{SE},\mathcal{R}^{NW}$ extend naturally to partial permutations. For any point $(a,b)\in\Z^2$ we can apply $\mathcal{R}^{SE}$ on the region of $\sigma$ that is southeast of $(a,b)$ and $\mathcal{R}^{NW}$ on the remaining region to get two pairs in $X_n$ by the above identifications. These will turn out to be component-wise disjoint and taking their component-wise union produces another pair in $X_n$. See also Example \ref{ex:mixrsk} below. Letting the point $(a,b)$ vary over $\Z^2$ for a fixed $\sigma$ will then produce all pre-images of $\sigma$ under $\psi$. Making this precise and proving it is the content of this section. 
		
		\begin{defi}\label{def:mixrsk}
			Let $\sigma\in S_n^{321}$ and $\Gamma=(a,b)\in\Z^2$. Set 
			\begin{align*}
				\sigma_{\ngeq \Gamma}&=\{(i,\sigma(i))\colon i<a\text{ or }\sigma(i)<b\},\text{ and }\\
				\sigma_{\geq \Gamma}&=\{(i,\sigma(i))\colon i\geq a,\sigma(i)\geq b\}.
			\end{align*}
			With $\mathcal{R}^{NW}(\sigma_{\ngeq \Gamma})=(I^N,J^N)$ and $\mathcal{R}^{SE}(\sigma_{\geq\Gamma})=(I^S,J^S)$ we define
			\begin{align*}
				\mathcal{R}_\Gamma(\sigma)=(I^N\cup I^S,J^N\cup J^S).
			\end{align*}
			We call $\mathcal{R}_\Gamma$ a \textit{mixed} RSK map.
		\end{defi}
		The condition $i\geq a,j\geq b$ can also be phrased visually as $(i,j)$ being southeast of $(a,b)$ in our shadow line setup. So southeast of $(a,b)$ we apply $\mathcal{R}^{SE}$ and in the remaining region we apply $\mathcal{R}^{NW}$. This is the reason for the indexing of these maps.
		\begin{rmk}
			The $321$-avoidance of $\sigma$ implies that $I^N\cap I^S=J^N\cap J^S=\emptyset$ so the unions in the definition of $\mathcal{R}_\Gamma(\sigma)$ are actually disjoint. In fact it is not hard to see that we even have $\max I^N < \min I^S$, if these sets are non-empty, and the same for $J^N,J^S$ for $321$-avoiding $\sigma$. In particular this shows that $\mathcal{R}_\Gamma(\sigma)\in X_n$ and so we get a map $\mathcal{R}_\Gamma: S_n^{321}\rightarrow X_n$. We will show in Lemma \ref{lem:mixrsk} that $\mathcal{R}_\Gamma$ is a section of $\psi$ for any $\Gamma\in\Z^2$.
			
			If $\Gamma=(n+1,n+1)$, then for all $\sigma\in S_n^{321}$ we have $\sigma_{\geq \Gamma}=\emptyset$ and $\sigma_{\ngeq \Gamma}=\sigma$, so
			\begin{align*}
				\mathcal{R}_{(n+1,n+1)}=\mathcal{R}^{NW}.
			\end{align*}
			Similarly, we also have $\mathcal{R}_{(1,1)}=\mathcal{R}^{SE}$. Thus both of the 'pure' orientation RSK maps that we introduced above occur as natural special cases of the mixed construction $\mathcal{R}_\Gamma$.
		\end{rmk}
		\begin{ex}\label{ex:mixrsk}
			We continue with the running example $\sigma=241579368$ and take as cut-off $\Gamma=(4,5)$. Then 
			\begin{align*}
				\sigma_{\ngeq \Gamma}&=\{(1,2),(2,4),(3,1),(7,3) \}\quad\text{and}\\ \sigma_{\geq \Gamma}&=\{(4,5),(5,7),(6,9),(8,6),(9,8)\}
			\end{align*}
			are the partial permutations defined by $\Gamma$. As Figure \ref{fig:exmix} illustrates, we have 
			\begin{align*}
				\mathcal{R}^{NW}(\sigma_{\ngeq \Gamma})&=(\{1,2\},\{1,3\})\quad\text{and} \\ 
				\mathcal{R}^{SE}(\sigma_{\geq \Gamma})&=(\{4,5,6\},\{5,6,8\})
			\end{align*}
			and so
			\begin{align*}
				\mathcal{R}_{(4,5)}(241579368)=(\{1,2,4,5,6\},\{1,3,5,6,8\})\in X_{9}.
			\end{align*}
			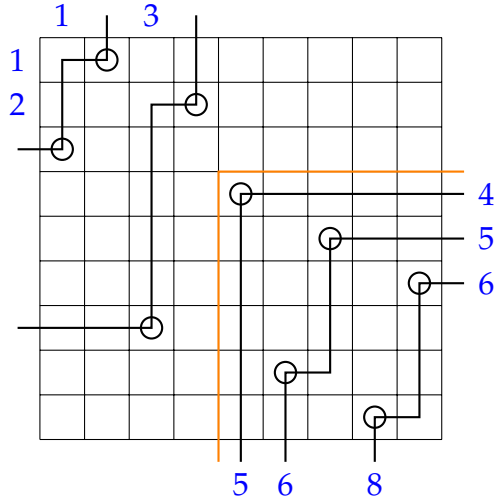
\begin{figure}[h!]
			\centering
			\begin{tikzpicture}[x=1.4em,y=1.4em, thick,color = black]
			\draw[step=1,black,thin] (0,0) grid (9,9);
			\draw [black] (1.5,8.5) circle (4pt);
			\draw [black] (3.5,7.5) circle (4pt);
    			\draw [black] (4.5,5.5) circle (4pt);
    			\draw [black] (6.5,4.5) circle (4pt);
    			\draw [black] (8.5,3.5) circle (4pt);
    			\draw [black] (0.5,6.5) circle (4pt);
    			\draw [black] (2.5,2.5) circle (4pt);
    			\draw [black] (5.5,1.5) circle (4pt);
    			\draw [black] (7.5,0.5) circle (4pt);
    			\draw [orange] (4,-0.5) -- (4,6) -- (9.5,6);
    			\draw [black] (-0.5,2.5) -- (2.5,2.5) -- (2.5,7.5) -- (3.5,7.5) -- (3.5,9.5);
    			\draw [black] (-0.5,6.5) -- (0.5,6.5) -- (0.5,8.5) -- (1.5,8.5) -- (1.5,9.5);
    			\draw [black] (4.5,-0.5) -- (4.5,5.5) -- (9.5,5.5);
    			\draw [black] (5.5,-0.5) -- (5.5,1.5) -- (6.5,1.5) -- (6.5,4.5) -- (9.5,4.5);
    			\draw [black] (7.5,-0.5) -- (7.5,0.5) -- (8.5,0.5) -- (8.5,3.5) -- (9.5,3.5);
    			\node [blue] at (-0.5,8.5) {$1$};
    			\node [blue] at (0.5,9.5) {$1$};
    			\node [blue] at (-0.5,7.5) {$2$};
    			\node [blue] at (2.5,9.5) {$3$};
    			
    			\node [blue] at (4.5,-1) {$5$};
    			\node [blue] at (10,5.5) {$4$};
    			\node [blue] at (5.5,-1) {$6$};
    			\node [blue] at (10,4.5) {$5$};
    			\node [blue] at (7.5,-1) {$8$};
    			\node [blue] at (10,3.5) {$6$};
			\end{tikzpicture}
			\caption{Illustrating the computation of $\mathcal{R}_{(4,5)}(241579368)$.}
			\label{fig:exmix}
		\end{figure}
    			In Figure \ref{fig:exmix} the orange line indicates the cut-off $\Gamma=(4,5)$: south-east of the orange line we apply $\mathcal{R}^{SE}$ and on the other side we apply $\mathcal{R}^{NW}$. In blue are the elements of the sets $I,J$ that each part of the construction yields, where the numbers indexing rows form $I$ and those indexing columns form $J$.
    			
			Identifying everything again with multisegments, we see that $\mathcal{R}_{(4,5)}(241579368)$ corresponds precisely to the pair of ladders $(\ml_1,\ml_2)$ considered in Example \ref{ex:Psi} and we have already shown in Figure \ref{fig:expsi} that $\Psi(\ml_1,\ml_2)=m_\sigma$. Thus, with the mixed RSK map $\mathcal{R}_{(4,5)}$ we have constructed another preimage of $\sigma$ under the map $\psi$. The following lemma shows that this holds in general and that all preimages can be found in this way.

		\end{ex}

In the following lemma we establish that $\mathcal{R}_\Gamma$ is always a section of $\psi$. For this we need to compare the cut-off $\Gamma$ in $\mathcal{R}_\Gamma$ with cut-offs in the map $\Psi$ applied to pairs of ladder multisegments. For this keep in mind that going from points $(a,b)\in\Z^2$ to a segment, we apply the shift $[a,b+n-1]$ in this section. 

		\begin{lem}\label{lem:mixrsk}
			Let $\sigma\in S_n^{321}$ be a $321$-avoiding permutation.
			\begin{enumerate}[a)]
				\item  If $\Gamma\in\Z^2$ is arbitrary, then either $\mathcal{R}_\Gamma(\sigma)=\mathcal{R}^{NW}(\sigma)$ or there is $i\in[n]$ such that $\mathcal{R}_\Gamma(\sigma)=\mathcal{R}_{(i,\sigma(i))}(\sigma)$. So,
				\begin{align*}
					\{\mathcal{R}_\Gamma(\sigma)\colon \Gamma\in\Z^2\}=\{\mathcal{R}_\Gamma(\sigma)\colon \Gamma=(i,\sigma(i))\text{ for some }i\text{ or }\Gamma=(n+1,n+1)\}.
				\end{align*}
				\item If $\Gamma=(i,\sigma(i))$ or $\Gamma=(n+1,n+1)$, and $(I,J)=\mathcal{R}_\Gamma(\sigma)$, then
				\begin{align*}
					\Psi(\ml_{I,J},\ml_{I^c,J^c})=\m_\sigma,
				\end{align*}
				in other words, $\psi(\mathcal{R}_\Gamma(\sigma))=\sigma$.
				\item 
				All pairs $\mathcal{R}_\Gamma(\sigma)$ with $\Gamma=(i,\sigma(i))$ for some $i$ or $\Gamma=(n+1,n+1)$ are distinct.
			\end{enumerate}
		\end{lem}
		\begin{proof}
			\begin{enumerate}[a)]
			\item 
			It is clear that $\mathcal{R}_\Gamma(\sigma)$ only depends on the set $\sigma_{\geq\Gamma}$ and not on the precise position $\Gamma\in\Z^2$. Let $D(\sigma)=\{(i,j)\in[n]^2\colon j<\sigma(i),\,i<\sigma^{-1}(j)\}$ denote the diagram of $\sigma$. Then it is easy to see that as $\Gamma$ varies over the set $\widetilde{D}(\sigma)\coloneqq D(\sigma)\cup\{(i,\sigma(i))\colon i\in[n]\}\cup\{(n+1,n+1)\}$, all possibilities for the set $\sigma_{\geq\Gamma}$ for $\Gamma\in\Z^2$ will occur. We prove the claim by showing that for all $\Gamma\in D(\sigma)$ there is $\gamma\in\{(i,\sigma(i))\colon i\in[n]\}$ or $\gamma=(n+1,n+1)$ such that $\mathcal{R}_\Gamma(\sigma)=\mathcal{R}_\gamma(\sigma)$.

			Let $\Gamma=(a,b)\in D(\sigma)$ and consider the constellation of shadow lines drawn in the computation of $\mathcal{R}_{\Gamma}(\sigma)$. If all shadow lines se of $\Gamma$ have an elbow in them, set $\gamma=(n+1,n+1)$. Otherwise, there is a nw-most shadow line that is se of $\Gamma$ that has no elbow in it. Necessarily, this shadow line then has a corner at a point $\gamma\coloneqq(i,\sigma(i))$ for some $i$. Once we show that $\mathcal{R}_{\Gamma}(\sigma)=\mathcal{R}_{\gamma}(\sigma)$, the proof of a) is done.
			
			By construction, the shadow line through $\gamma=(i,\sigma(i))$ does not change when moving the cut-off from $\Gamma$ to $\gamma$ and the same holds for all shadow lines that are SE of $\gamma$ or NW of $\Gamma$. 
			
			It thus remains to show that, while the shadow lines between $\Gamma$ and $\gamma$ change orientation, the image of $\sigma$ under the mixed-RSK maps does not change when moving the cut-off from $\Gamma$ to $\gamma$. By construction of $\gamma$, all shadow lines in the SE-region of $\mathcal{R}_\Gamma$ strictly between $\Gamma$ and $\gamma$ have an elbow in them and so contain two points of $\sigma$ each. In the computation of $\mathcal{R}_\gamma$, these points are now in the NW-region but will still be matched in the same pairs by shadow lines. To see this, index the shadow lines in $\mathcal{R}^{SE}(\sigma_{\geq\Gamma})$ that are strictly nw of $\gamma$ by $L_1,\dots,L_k$ going from northwest to southeast. Denote by $(i_j,i_j^\prime)$, $i_j>i_j^\prime$ the row indices of the points of shadow line $L_j$ for $j=1,\dots,k$. We show that in $\mathcal{R}^{NW}(\sigma_{\ngeq\gamma})$, the points corresponding to row indices $i_j,i_j^\prime$ still lie on the same shadow line by induction. By construction, $i_k$ is the maximal index of a row that contains a point of $\sigma_{\ngeq\gamma}$, so the first shadow line of $\mathcal{R}^{NW}(\sigma_{\ngeq\gamma})$ passes through this point. The only available points for this shadow line to connect to are in row $i_k^\prime$, in rows $i_j,i_j^\prime$ for $j<k$ and the points in $\sigma_{\ngeq\Gamma}$. Since $(i_k,\sigma(i_k))$ is se of $\Gamma$, the latter cannot occur since it would have already connected to the available point $(i_k^\prime,\sigma(i_k^\prime))$ and by $321$-avoidance of $\sigma$, there cannot be more than $2$ points on any shadow line. So the line entering in row $i_k$ connects to a point in some row $i_j,i_j^\prime$ with $j<k$ or to $i_k^\prime$. If it connected to a point with $j<k$, this would contradict the disjointness of shadow lines in $\mathcal{R}^{SE}(\sigma_{\geq\Gamma})$ and so it can only connect to its original partner $(i_k^\prime,\sigma(i_k^\prime))$. Continuing this process for $j<k$ shows that shadow lines still match the same pairs between $\Gamma$ and $\gamma$ in $\mathcal{R}_\gamma(\sigma)$ as claimed.
			
			 Since the orientation of the lines in this region changed, what were the row/column coordinates of the half-infinite rays in $\mathcal{R}_\Gamma$ now become coordinates of elbows in $\mathcal{R}_\gamma$ and vice-versa. Recalling the identifications between (inverted) tableaux and subsets of $[n]$ that are used in the definition of the mixed-RSK maps (see Figure \ref{fig:exmix}) then shows that in total $\mathcal{R}_\gamma(\sigma)=\mathcal{R}_\Gamma(\sigma)$. Figure \ref{fig:proof} illustrates this argument.
			
				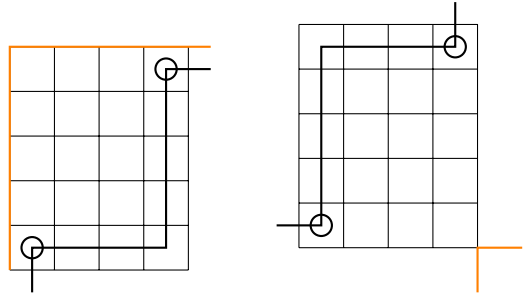
\begin{figure}
					\centering
					\begin{tikzpicture}[x=1.4em,y=1.4em, thick,color = black]
						\draw[step=1,black,thin](0,0) grid (4,5);
						\draw (3.5,4.5) circle (4pt);
						\draw (0.5,0.5) circle (4pt);
						\draw [orange] (0,-0.0) -- (0,5) -- (4.5,5);
						\draw (0.5,-0.5) -- (0.5,0.5) -- (3.5,0.5) -- (3.5,4.5) -- (4.5,4.5);
					\end{tikzpicture}\quad\quad
					\begin{tikzpicture}[x=1.4em,y=1.4em, thick,color = black]
						\draw[step=1,black,thin](0,0) grid (4,5);
						\draw (3.5,4.5) circle (4pt);
						\draw (0.5,0.5) circle (4pt);
						\draw [orange]  (4,-1) -- (4,0) -- (5,0);
						\draw (-0.5,0.5) -- (0.5,0.5) -- (0.5,4.5) -- (3.5,4.5) -- (3.5,5.5);
					\end{tikzpicture}
					\caption{Illustrating the effect of moving the cut-off in the diagram of $\sigma$.}
					\label{fig:proof}
					\end{figure}
			
				\item The map $\Psi$ is constructed precisely in such a way that, in the region that is southeast of $\Gamma(\ml_1,\ml_2)$, it is inverse to $\mathcal{R}^{SE}$ and in the remaining region it is inverse to $\mathcal{R}^{NW}$ (cf. \cite{GL21} Proposition 2.4.) and $\Psi(\ml_1,\ml_2)$ is the sum of these inverses. Thus it suffices to show that $\Gamma(\ml_{I,J},\ml_{I^c,J^c})=\Gamma$, i.e. that the pair of ladders produced by a mixed RSK map has the same cut off as the mixed RSK map.
				
				By a) we may assume that $\Gamma=(n+1,n+1)$ or $\Gamma=(i,\sigma(i))$ for some $i$. If $\Gamma=(n+1,n+1)$ then the claim follows from Proposition \ref{prop:nwrsk}. Thus we now assume $\Gamma=(i,\sigma(i))$ for some $i$. For notational simplicity set $\ml=\ml_{I,J}$ and $\ml^c=\ml_{I^c,J^c}$. By the definition of $\mathcal{R}_\Gamma(\sigma)=(I,J)$ it is clear that $\Gamma$ corresponds to a segment of $\ml$ and that $(\ml_{\ngeq \Gamma},\ml^c_{\ngeq \Gamma})$ is the pair of ladders corresponding to the image of $\mathcal{R}^{NW}(\sigma_{\ngeq \Gamma})$. Since any image under $\mathcal{R}^{NW}$ is dominant (this is just a reformulation of the monotonicity conditions on SYT), we see that $(\ml_{\ngeq \Gamma},\ml^c_{\ngeq \Gamma})$ is dominant and so $\Gamma(\ml,\ml^c)\geq \Gamma$. 
				
				To prove the claim, we need to show that for any segment $\Delta_\gamma>\Gamma$ of $\ml$, the pair $(\ml_{\ngeq\Delta_\gamma},\ml^c_{\ngeq\Delta_\gamma})$ is not dominant and also that the whole pair $(\ml,\ml^c)$ is not dominant. The segments $\Gamma+\dots+\Delta_{\gamma+1}+\Delta_\gamma$ of $\ml$ correspond to shadow lines in $\mathcal{R}^{SE}(\sigma_{\geq\Gamma})$. Going through these shadow lines from southeast to northwest, if they contain an elbow this produces a segment $\Delta^\prime$ of $\ml^c$, corresponding to the elbow, and a segment $\Delta$ of $\ml$ corresponding to the infinite rays. Clearly then $\Delta\leq\Delta^\prime$. If a shadow line contains no elbow, it produces only a segment of $\ml$. Thus, since we already know that $(\ml_{\ngeq\Gamma},\ml^c_{\ngeq\Gamma})$ is dominant, $(\ml_{\ngeq\Delta_\gamma},\ml^c_{\ngeq\Delta_\gamma})$ is dominant if and only if all shadow lines between $\Gamma$ and $\Delta_\gamma$ have an elbow, since otherwise $\ml^c$ has too few segments to dominate $\ml$ in this region. But by construction, the line going through $\Gamma$ does not have one and so $(\ml_{\ngeq\Delta_\gamma},\ml^c_{\ngeq\Delta_\gamma})$ is not dominant. The same argument also shows that $(\ml,\ml^c)$ itself cannot be dominant.
				
				\item 
				\sloppypar In the proof of b) we have shown that $\Gamma(\ml_{I,J},\ml_{I^c,J^c})=\Gamma$ if $(I,J)=\mathcal{R}_\Gamma(\sigma)$ for $\Gamma=(i,\sigma(i))$ and if $\Gamma=(n+1,n+1)$, then $(\ml,\ml^c)$ corresponding to $\mathcal{R}_\Gamma(\sigma)$ is dominant. So all elements of 
				$ \{\mathcal{R}_\Gamma(\sigma)\colon \Gamma=(i,\sigma(i))\text{ for some }i\text{ or }\Gamma=(n+1,n+1)\} $
				correspond to pairs of ladders  $(\ml,\ml^c)$ with different cut-offs $\Gamma(\ml,\ml^c)$ and are therefore distinct.
			\end{enumerate}
		\end{proof}

		\begin{proof}[Proof of Theorem \ref{thm:n+1}]
\sloppypar			Let $\sigma\in S_n^{321}$. Then a) and b) of the previous lemma show $\{\mathcal{R}_\Gamma(\sigma)\colon \Gamma\in\Z^2\}\subseteq\psi^{-1}(\sigma)$ and c) shows that $|\{\mathcal{R}_\Gamma(\sigma)\colon \Gamma\in\Z^2\}|\geq n+1$. Since this holds for all $\sigma\in S_n^{321}$ and $|X_n|=(n+1)|S_n^{321}|$ we must have $|\psi^{-1}(\sigma)|=n+1$. In fact this also shows $\psi^{-1}(\sigma)=\{\mathcal{R}_\Gamma(\sigma)\colon \Gamma=(i,\sigma(i))\text{ for some }i\text{ or }\Gamma=(n+1,n+1)\}$. 
		\end{proof}
		\begin{ex}\label{ex:full}
			We illustrate this parametrisation of the set $\psi^{-1}(\sigma)$ with our running example.
			
			Let $\sigma=241579368\in S_9^{321}$. We have already computed $\mathcal{R}^{NW}(\sigma)$ resp. $\mathcal{R}^{SE}(\sigma)$ in Example \ref{ex:rsknw} resp. Example \ref{ex:rskse}. Below we also give the shadow line constructions computing $\mathcal{R}_\Gamma(\sigma)$ for the nine positions of $\Gamma$ corresponding to points of the permutation $\sigma$. By the previous discussion, this computes $\psi^{-1}(\sigma)$, i.e. it computes all regular pairs of ladders $(\ml_1,\ml_2)$ such that $\soc(\ml_1,\ml_2)=\m_\sigma$. As before, we indicate the cut-off $\Gamma$ by an orange line and we write $\mathcal{R}_\Gamma(\sigma)=(I,J)$ below each picture. Recall from Definition \ref{def:mixrsk} and the constructions in Section \ref{sec:rsk} that $I$ resp. $J$ are the row resp. column indices of elbows northwest of $\Gamma$ and half-infinite shadow lines southeast of $\Gamma$.
			
		\noindent\makebox[\textwidth]{
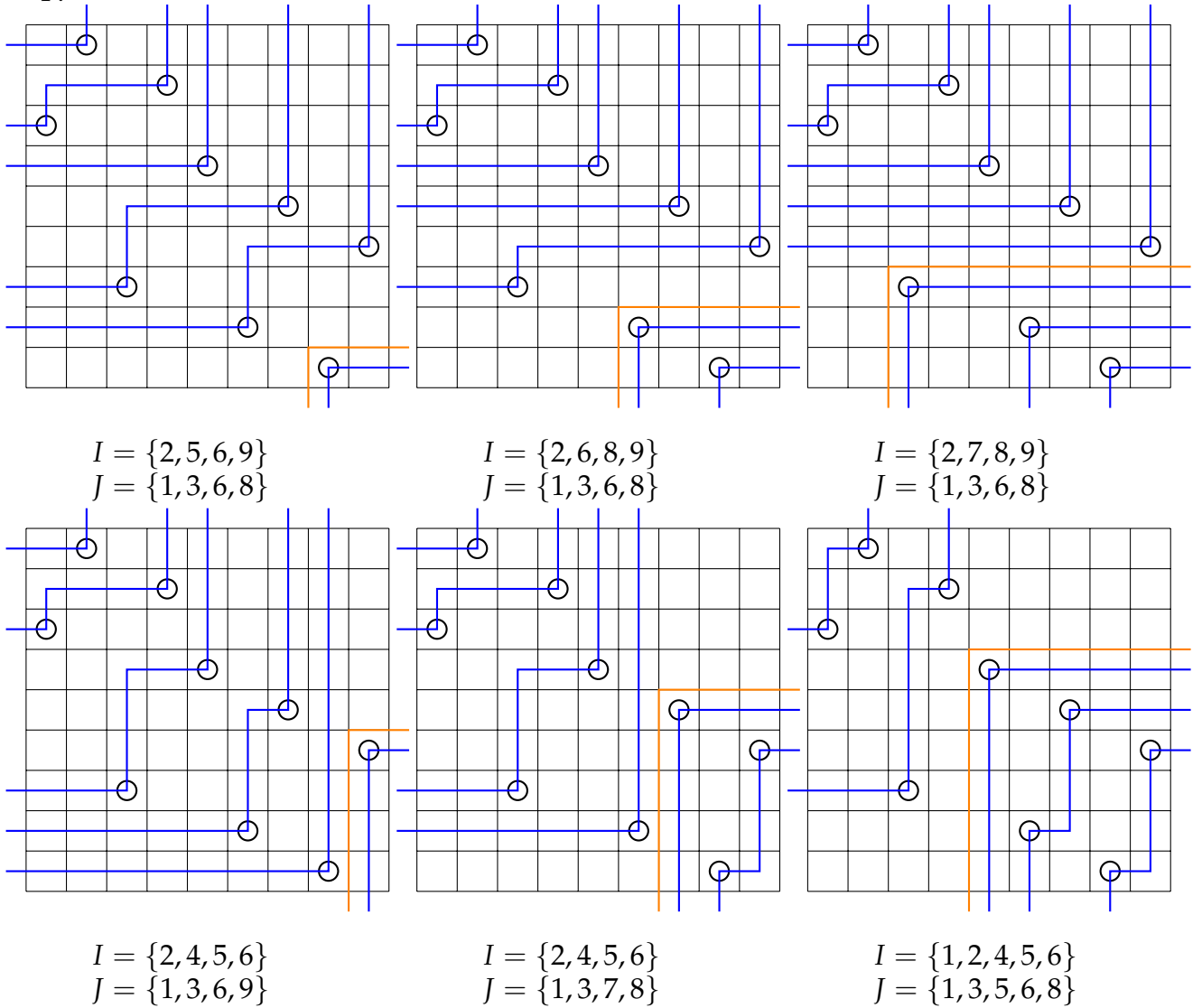
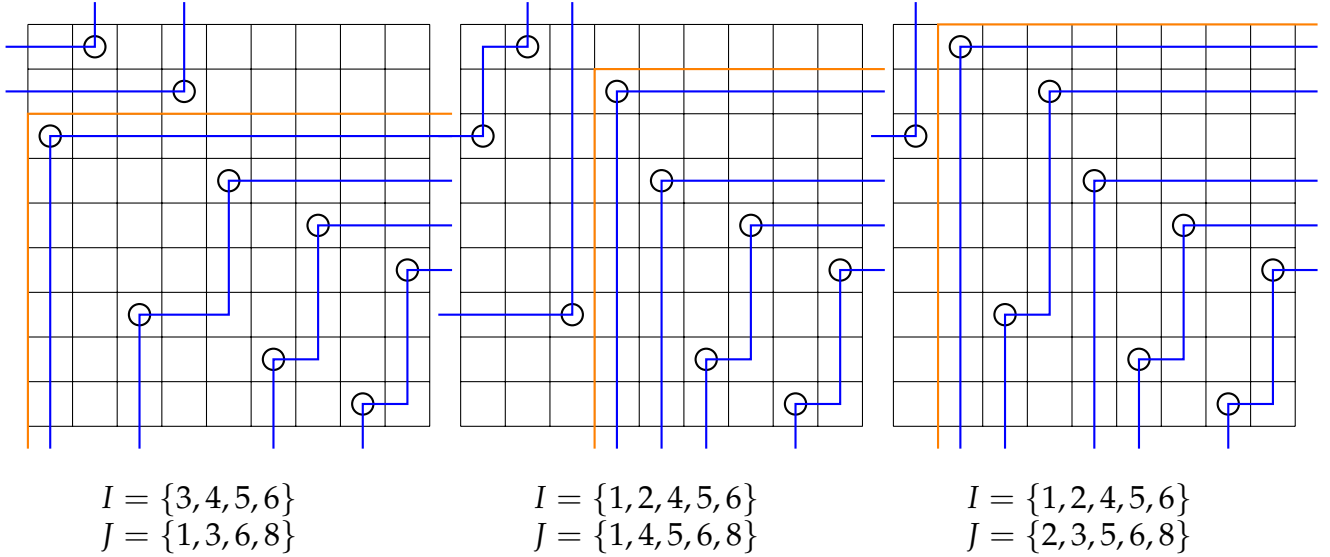
		\begin{minipage}{0.35\textwidth}
				\begin{tikzpicture}[x=1.4em,y=1.4em, thick,color = black]
			\draw[step=1,black,thin] (0,0) grid (9,9);
			\draw [black] (1.5,8.5) circle (4pt);
			\draw [black] (3.5,7.5) circle (4pt);
    			\draw [black] (4.5,5.5) circle (4pt);
    			\draw [black] (6.5,4.5) circle (4pt);
    			\draw [black] (8.5,3.5) circle (4pt);
    			\draw [black] (0.5,6.5) circle (4pt);
    			\draw [black] (2.5,2.5) circle (4pt);
    			\draw [black] (5.5,1.5) circle (4pt);
    			\draw [black] (7.5,0.5) circle (4pt);
    			
    			\draw [orange] (7,-0.5) -- (7,1) -- (9.5,1);
    			
    			\draw [blue] (7.5,-0.5) -- (7.5,0.5) -- (9.5,0.5);
    			
    			\draw [blue] (-0.5,1.5) -- (5.5,1.5) -- (5.5,3.5) -- (8.5,3.5) -- (8.5,9.5);
    			\draw [blue] (-0.5,2.5) -- (2.5,2.5) -- (2.5,4.5) -- (6.5,4.5) -- (6.5,9.5);
    			\draw [blue] (-0.5,5.5) -- (4.5,5.5) -- (4.5,9.5);
    			\draw [blue] (-0.5,6.5) -- (0.5,6.5) -- (0.5,7.5) -- (3.5,7.5) -- (3.5,9.5);
    			\draw [blue] (-0.5,8.5) -- (1.5,8.5) -- (1.5,9.5);
		\end{tikzpicture}
    				\centering
    				\captionof*{figure}{$I=\{2,5,6,9\}$\\$J=\{1,3,6,8\}$}
		\end{minipage}
		\begin{minipage}{0.35\textwidth}
				\begin{tikzpicture}[x=1.4em,y=1.4em, thick,color = black]
			\draw[step=1,black,thin] (0,0) grid (9,9);
			\draw [black] (1.5,8.5) circle (4pt);
			\draw [black] (3.5,7.5) circle (4pt);
    			\draw [black] (4.5,5.5) circle (4pt);
    			\draw [black] (6.5,4.5) circle (4pt);
    			\draw [black] (8.5,3.5) circle (4pt);
    			\draw [black] (0.5,6.5) circle (4pt);
    			\draw [black] (2.5,2.5) circle (4pt);
    			\draw [black] (5.5,1.5) circle (4pt);
    			\draw [black] (7.5,0.5) circle (4pt);
    			
    			\draw [orange] (5,-0.5) -- (5,2) -- (9.5,2);
    			
    			\draw [blue] (7.5,-0.5) -- (7.5,0.5) -- (9.5,0.5);
    			\draw [blue] (5.5,-0.5) -- (5.5,1.5) -- (9.5,1.5);
    			
    			\draw [blue] (-0.5,2.5) -- (2.5,2.5) -- (2.5,3.5) -- (8.5,3.5) -- (8.5,9.5);
    			\draw [blue] (-0.5,4.5) -- (6.5,4.5) -- (6.5,9.5);
    			\draw [blue] (-0.5,5.5) -- (4.5,5.5) -- (4.5,9.5);
    			\draw [blue] (-0.5,6.5) -- (0.5,6.5) -- (0.5,7.5) -- (3.5,7.5) -- (3.5,9.5);
    			\draw [blue] (-0.5,8.5) -- (1.5,8.5) -- (1.5,9.5);
		\end{tikzpicture}
    				\centering
    				\captionof*{figure}{$I=\{2,6,8,9\}$\\$J=\{1,3,6,8\}$}
		\end{minipage}
		\begin{minipage}{0.35\textwidth}
				\begin{tikzpicture}[x=1.4em,y=1.4em, thick,color = black]
			\draw[step=1,black,thin] (0,0) grid (9,9);
			\draw [black] (1.5,8.5) circle (4pt);
			\draw [black] (3.5,7.5) circle (4pt);
    			\draw [black] (4.5,5.5) circle (4pt);
    			\draw [black] (6.5,4.5) circle (4pt);
    			\draw [black] (8.5,3.5) circle (4pt);
    			\draw [black] (0.5,6.5) circle (4pt);
    			\draw [black] (2.5,2.5) circle (4pt);
    			\draw [black] (5.5,1.5) circle (4pt);
    			\draw [black] (7.5,0.5) circle (4pt);
    			
    			\draw [orange] (2,-0.5) -- (2,3) -- (9.5,3);
    			
    			\draw [blue] (7.5,-0.5) -- (7.5,0.5) -- (9.5,0.5);
    			\draw [blue] (5.5,-0.5) -- (5.5,1.5) -- (9.5,1.5);
    			\draw [blue] (2.5,-0.5) -- (2.5,2.5) -- (9.5,2.5);
    			
    			\draw [blue] (-0.5,3.5) -- (8.5,3.5) -- (8.5,9.5);
    			\draw [blue] (-0.5,4.5) -- (6.5,4.5) -- (6.5,9.5);
    			\draw [blue] (-0.5,5.5) -- (4.5,5.5) -- (4.5,9.5);
    			\draw [blue] (-0.5,6.5) -- (0.5,6.5) -- (0.5,7.5) -- (3.5,7.5) -- (3.5,9.5);
    			\draw [blue] (-0.5,8.5) -- (1.5,8.5) -- (1.5,9.5);
		\end{tikzpicture}
    				\centering
    				\captionof*{figure}{$I=\{2,7,8,9\}$\\$J=\{1,3,6,8\}$}
		\end{minipage}
		}
		
		\noindent\makebox[\textwidth]{
		\begin{minipage}{0.35\textwidth}
				\begin{tikzpicture}[x=1.4em,y=1.4em, thick,color = black]
			\draw[step=1,black,thin] (0,0) grid (9,9);
			\draw [black] (1.5,8.5) circle (4pt);
			\draw [black] (3.5,7.5) circle (4pt);
    			\draw [black] (4.5,5.5) circle (4pt);
    			\draw [black] (6.5,4.5) circle (4pt);
    			\draw [black] (8.5,3.5) circle (4pt);
    			\draw [black] (0.5,6.5) circle (4pt);
    			\draw [black] (2.5,2.5) circle (4pt);
    			\draw [black] (5.5,1.5) circle (4pt);
    			\draw [black] (7.5,0.5) circle (4pt);
    			
    			\draw [orange] (8,-0.5) -- (8,4) -- (9.5,4);
    			
    			\draw [blue] (8.5,-0.5) -- (8.5,3.5) -- (9.5,3.5);
    			
    			\draw [blue] (-0.5,0.5) -- (7.5,0.5) -- (7.5,9.5);
    			\draw [blue] (-0.5,1.5) -- (5.5,1.5) -- (5.5,4.5) -- (6.5,4.5) -- (6.5,9.5);	
    			\draw [blue] (-0.5,2.5) -- (2.5,2.5) -- (2.5,5.5) -- (4.5,5.5) -- (4.5,9.5);
    			\draw [blue] (-0.5,6.5) -- (0.5,6.5) -- (0.5,7.5) -- (3.5,7.5) -- (3.5,9.5);
    			\draw [blue] (-0.5,8.5) -- (1.5,8.5) -- (1.5,9.5);
		\end{tikzpicture}
    				\centering
    				\captionof*{figure}{$I=\{2,4,5,6\}$\\$J=\{1,3,6,9\}$}
		\end{minipage}
		\begin{minipage}{0.35\textwidth}
				\begin{tikzpicture}[x=1.4em,y=1.4em, thick,color = black]
			\draw[step=1,black,thin] (0,0) grid (9,9);
			\draw [black] (1.5,8.5) circle (4pt);
			\draw [black] (3.5,7.5) circle (4pt);
    			\draw [black] (4.5,5.5) circle (4pt);
    			\draw [black] (6.5,4.5) circle (4pt);
    			\draw [black] (8.5,3.5) circle (4pt);
    			\draw [black] (0.5,6.5) circle (4pt);
    			\draw [black] (2.5,2.5) circle (4pt);
    			\draw [black] (5.5,1.5) circle (4pt);
    			\draw [black] (7.5,0.5) circle (4pt);
    			
    			\draw [orange] (6,-0.5) -- (6,5) -- (9.5,5);
    			
    			\draw [blue] (6.5,-0.5) -- (6.5,4.5) -- (9.5,4.5);
    			\draw [blue] (7.5,-0.5) -- (7.5,0.5) -- (8.5,0.5) -- (8.5,3.5) -- (9.5,3.5);

    			\draw [blue] (-0.5,1.5) -- (5.5,1.5) -- (5.5,9.5);	
    			\draw [blue] (-0.5,2.5) -- (2.5,2.5) -- (2.5,5.5) -- (4.5,5.5) -- (4.5,9.5);
    			\draw [blue] (-0.5,6.5) -- (0.5,6.5) -- (0.5,7.5) -- (3.5,7.5) -- (3.5,9.5);
    			\draw [blue] (-0.5,8.5) -- (1.5,8.5) -- (1.5,9.5);
		\end{tikzpicture}
    				\centering
    				\captionof*{figure}{$I=\{2,4,5,6\}$\\$J=\{1,3,7,8\}$}
		\end{minipage}
		\begin{minipage}{0.35\textwidth}
				\begin{tikzpicture}[x=1.4em,y=1.4em, thick,color = black]
			\draw[step=1,black,thin] (0,0) grid (9,9);
			\draw [black] (1.5,8.5) circle (4pt);
			\draw [black] (3.5,7.5) circle (4pt);
    			\draw [black] (4.5,5.5) circle (4pt);
    			\draw [black] (6.5,4.5) circle (4pt);
    			\draw [black] (8.5,3.5) circle (4pt);
    			\draw [black] (0.5,6.5) circle (4pt);
    			\draw [black] (2.5,2.5) circle (4pt);
    			\draw [black] (5.5,1.5) circle (4pt);
    			\draw [black] (7.5,0.5) circle (4pt);
    			
    			\draw [orange] (4,-0.5) -- (4,6) -- (9.5,6);
    			
    			\draw [blue] (5.5,-0.5) -- (5.5,1.5) -- (6.5,1.5) -- (6.5,4.5) -- (9.5,4.5);
    			\draw [blue] (7.5,-0.5) -- (7.5,0.5) -- (8.5,0.5) -- (8.5,3.5) -- (9.5,3.5);
    			\draw [blue] (4.5,-0.5) -- (4.5,5.5) -- (9.5,5.5);
	
    			\draw [blue] (-0.5,2.5) -- (2.5,2.5) -- (2.5,7.5) -- (3.5,7.5) -- (3.5,9.5);
    			\draw [blue] (-0.5,6.5) -- (0.5,6.5) -- (0.5,8.5) -- (1.5,8.5) -- (1.5,9.5);
		\end{tikzpicture}
    				\centering
    				\captionof*{figure}{$I=\{1,2,4,5,6\}$\\$J=\{1,3,5,6,8\}$}
		\end{minipage}
		}
		
		\noindent\makebox[\textwidth]{
		\begin{minipage}{0.35\textwidth}
				\begin{tikzpicture}[x=1.4em,y=1.4em, thick,color = black]
			\draw[step=1,black,thin] (0,0) grid (9,9);
			\draw [black] (1.5,8.5) circle (4pt);
			\draw [black] (3.5,7.5) circle (4pt);
    			\draw [black] (4.5,5.5) circle (4pt);
    			\draw [black] (6.5,4.5) circle (4pt);
    			\draw [black] (8.5,3.5) circle (4pt);
    			\draw [black] (0.5,6.5) circle (4pt);
    			\draw [black] (2.5,2.5) circle (4pt);
    			\draw [black] (5.5,1.5) circle (4pt);
    			\draw [black] (7.5,0.5) circle (4pt);
    			
    			\draw [orange] (0,-0.5) -- (0,7) -- (9.5,7);
    			
    			\draw [blue] (0.5,-0.5) -- (0.5,6.5) -- (9.5,6.5);
    			\draw [blue] (2.5,-0.5) -- (2.5,2.5) -- (4.5,2.5) -- (4.5,5.5) -- (9.5,5.5);
    			\draw [blue] (5.5,-0.5) -- (5.5,1.5) -- (6.5,1.5) -- (6.5,4.5) -- (9.5,4.5);
    			\draw [blue] (7.5,-0.5) -- (7.5,0.5) -- (8.5,0.5) -- (8.5,3.5) -- (9.5,3.5);
	
    			\draw [blue] (-0.5,7.5) -- (3.5,7.5) -- (3.5,9.5);
    			\draw [blue] (-0.5,8.5) -- (1.5,8.5) -- (1.5,9.5);
		\end{tikzpicture}
    				\centering
    				\captionof*{figure}{$I=\{3,4,5,6\}$\\$J=\{1,3,6,8\}$}
		\end{minipage}
		\begin{minipage}{0.35\textwidth}
				\begin{tikzpicture}[x=1.4em,y=1.4em, thick,color = black]
			\draw[step=1,black,thin] (0,0) grid (9,9);
			\draw [black] (1.5,8.5) circle (4pt);
			\draw [black] (3.5,7.5) circle (4pt);
    			\draw [black] (4.5,5.5) circle (4pt);
    			\draw [black] (6.5,4.5) circle (4pt);
    			\draw [black] (8.5,3.5) circle (4pt);
    			\draw [black] (0.5,6.5) circle (4pt);
    			\draw [black] (2.5,2.5) circle (4pt);
    			\draw [black] (5.5,1.5) circle (4pt);
    			\draw [black] (7.5,0.5) circle (4pt);
    			
    			\draw [orange] (3,-0.5) -- (3,8) -- (9.5,8);
    			
    			\draw [blue] (3.5,-0.5) -- (3.5,7.5) -- (9.5,7.5);
    			\draw [blue] (4.5,-0.5) -- (4.5,5.5) -- (9.5,5.5);
    			\draw [blue] (5.5,-0.5) -- (5.5,1.5) -- (6.5,1.5) -- (6.5,4.5) -- (9.5,4.5);
    			\draw [blue] (7.5,-0.5) -- (7.5,0.5) -- (8.5,0.5) -- (8.5,3.5) -- (9.5,3.5);
	
			\draw [blue] (-0.5,2.5) -- (2.5,2.5) -- (2.5,9.5);
    			\draw [blue] (-0.5,6.5) -- (0.5,6.5) -- (0.5,8.5) -- (1.5,8.5) -- (1.5,9.5);
		\end{tikzpicture}
    				\centering
    				\captionof*{figure}{$I=\{1,2,4,5,6\}$\\$J=\{1,4,5,6,8\}$}
		\end{minipage}
		\begin{minipage}{0.35\textwidth}
				\begin{tikzpicture}[x=1.4em,y=1.4em, thick,color = black]
			\draw[step=1,black,thin] (0,0) grid (9,9);
			\draw [black] (1.5,8.5) circle (4pt);
			\draw [black] (3.5,7.5) circle (4pt);
    			\draw [black] (4.5,5.5) circle (4pt);
    			\draw [black] (6.5,4.5) circle (4pt);
    			\draw [black] (8.5,3.5) circle (4pt);
    			\draw [black] (0.5,6.5) circle (4pt);
    			\draw [black] (2.5,2.5) circle (4pt);
    			\draw [black] (5.5,1.5) circle (4pt);
    			\draw [black] (7.5,0.5) circle (4pt);
    			
    			\draw [orange] (1,-0.5) -- (1,9) -- (9.5,9);
    			
    			\draw [blue] (1.5,-0.5) -- (1.5,8.5) -- (9.5,8.5);
    			\draw [blue] (2.5,-0.5) -- (2.5,2.5) -- (3.5,2.5) -- (3.5,7.5) -- (9.5,7.5);;
    			\draw [blue] (4.5,-0.5) -- (4.5,5.5) -- (9.5,5.5);
    			\draw [blue] (5.5,-0.5) -- (5.5,1.5) -- (6.5,1.5) -- (6.5,4.5) -- (9.5,4.5);
    			\draw [blue] (7.5,-0.5) -- (7.5,0.5) -- (8.5,0.5) -- (8.5,3.5) -- (9.5,3.5);
	
    			\draw [blue] (-0.5,6.5) -- (0.5,6.5) -- (0.5,9.5);
		\end{tikzpicture}
    				\centering
    				\captionof*{figure}{$I=\{1,2,4,5,6\}$\\$J=\{2,3,5,6,8\}$}
		\end{minipage}
		}
		\end{ex}
		\begin{rmk}
			The results of this section hold somewhat more generally. Instead of considering $\mathcal{R}_\Gamma(\sigma)$ only for $321$-avoiding permutations $\sigma$, one can define equally well $\mathcal{R}_\Gamma(\m)$ for any multisegment $\m=\sum_{i\in I}\Delta_i$ satisfying $\b(\Delta_i)\leq \e(\Delta_j)$ for all $i,j\in I$ and that is '$321$-avoiding' in the sense that there are no pairwise distinct $i,j,k\in I$ such that $\Delta_i\subseteq\Delta_j\subseteq\Delta_k$. In this extension, it is important to choose the right version of RSK to deal with multiplicities. In the northwest region, this should be the map $\mathcal{K}$ of \cite{GL21} and in the southeast region one has to take the appropriate dual (in the sense of Lemma \ref{lem:dualdom}) version of it.
			
 			The mixed-RSK maps $\mathcal{R}_\Gamma$ are then still sections of $\Psi$. In particular, this shows how $\Psi$, as an inverse to our mixed RSK, extends the description of $\soc(Z(\ml_1)\times Z(\ml_2))$ for saturated $(\ml_1,\ml_2)$ (i.e. $\b(\Delta)\leq \e(\Delta^\prime)$ for all $\Delta,\Delta^\prime$ in $\ml_1,\ml_2$) of \cite{GL21} in terms of RSK combinatorics.
			
			It would be interesting to get rid of the pattern avoidance condition in this description and so get a simple combinatorial method to compute $\soc(Z(\m)\times Z(\ml))$ where $\m$ is not necessarily a ladder.
		\end{rmk}

\sloppy
\printbibliography
\end{document}